\documentclass[11pt]{amsart}

\usepackage{latexsym,amscd,amssymb,graphicx,color,amsthm,bm,amsmath,
  cancel,enumitem,ytableau}
\usepackage{tikz}
\usepackage[margin=1in]{geometry}
\usepackage{hyperref}
\usepackage[all,cmtip]{xy}
\usepackage{mathtools}
\usepackage{rotating}
\usetikzlibrary{math,arrows.meta}
\usetikzlibrary{calc}

\usepackage[backend=bibtex, style=alphabetic]{biblatex}
\usepackage{float}
\usetikzlibrary{intersections}
\DeclareEmphSequence{\color{blue}\itshape}

\numberwithin{equation}{section}
\newtheorem{theorem}{Theorem}[section]
\newtheorem{proposition}[theorem]{Proposition}
\newtheorem{corollary}[theorem]{Corollary}
\newtheorem{lemma}[theorem]{Lemma}
\newtheorem{conjecture}[theorem]{Conjecture}

\newtheorem{problem}[theorem]{Problem}
\newtheorem{example}[theorem]{Example}
\newtheorem{remark}[theorem]{Remark}
\newtheorem{definition}[theorem]{Definition}
\theoremstyle{definition}

\newcommand{\one}{{\mathbf{1}}}

\newcommand{\Frob}{{\mathrm{Frob}}}
\newcommand{\symm}{{\fS}}

\newcommand{\QQ}{{\mathbb{Q}}}
\newcommand{\ZZ}{{\mathbb{Z}}}

\newcommand{\FF}{\mathsf{F}}

\newcommand{\SYT}{\mathrm{SYT}}
\newcommand{\SSYT}{\mathrm{SSYT}}

\newcommand{\xxx}{{\mathbf{x}}}

\newcommand{\dschub}{\widetilde{S}}

\newcommand{\inv}{\operatorname{inv}}

\newcommand{\crSgn}{\mathrm{cr}^{\mathsf{sign}}}
\newcommand{\alSgn}{\mathrm{al}^{\mathsf{sign}}}
\newcommand{\wex}{\mathrm{wex}}

\newcommand{\numcolor}[1]{{\operatorname{neg}(#1)}}
\newcommand{\coloredfunc}[3]{{\Phi_{#1,#2,#3}}}
\newcommand{\EXC}[1]{{\operatorname{EXC}(#1)}}
\newcommand{\exc}[1]{{\operatorname{exc}(#1)}}
\newcommand{\DES}[1]{\operatorname{DES}(#1)}
\newcommand{\des}[1]{\operatorname{des}(#1)}
\newcommand{\DEX}[1]{\operatorname{DEX}(#1)}

\newcommand{\nega}[1]{\operatorname{neg}(#1)}
\newcommand{\al}[1]{\operatorname{al}^{\mathsf{col}}(#1)}
\newcommand{\cro}[1]{\operatorname{cr}^{\mathsf{col}}(#1)}

\newcommand{\fS}{\mathfrak{S}} 
\newcommand{\fB}{\mathfrak{B}} 
\newcommand{\coloredperm}[2]{{\fB_{#1,#2}}}
\definecolor{crprimed}{HTML}{2767B4}
\definecolor{crnonexc}{HTML}{17805B}
\definecolor{crbarred}{HTML}{C57816}
\definecolor{crunused}{HTML}{CDD3DB}

\newcommand{\CrDiskPanel}[1]{%
  \begin{tikzpicture}[font=\scriptsize,line cap=round,line join=round]
    \path[use as bounding box] (-3.15,-2.88) rectangle (3.15,2.88);
    \draw[crunused,line width=.45pt] (0,0) circle (2.25);
    \draw[crunused,densely dashed,line width=.4pt] (-2.25,0)--(2.25,0);
    \draw[crunused,->,>=Stealth] (150:1.91) arc (150:123:1.91);
    \draw[crunused,->,>=Stealth] (-25:1.91) arc (-25:-52:1.91);
    \foreach \crslot in {0,...,15} {
      \pgfmathsetmacro{\crturn}{180-(\crslot+.5)*360/16}
      \coordinate (crpt-\crslot) at (\crturn:2.25);
    }
    \ifnum#1=0
      \foreach \crslot in {0,...,7} {
        \pgfmathsetmacro{\crturn}{180-(\crslot+.5)*360/8}
        \pgfmathtruncatemacro{\crindex}{mod(\crslot,4)+1}
        \fill (\crturn:2.25) circle (1.8pt);
        \node[fill=white,inner sep=1pt] at (\crturn:2.7)
          {$\crindex$};
      }
    \else
      \ifnum#1>2
        \draw[crprimed,line width=.85pt] (crpt-3)--(crpt-4);
      \fi
      \ifnum#1>3
        \draw[crnonexc,line width=.85pt,name path=crfirstnon]
          (crpt-13)--(crpt-8);
      \fi
      \ifnum#1>4
        \draw[crnonexc,line width=.85pt,name path=crsecondnon]
          (crpt-15)--(crpt-10);
      \fi
      \ifnum#1>5
        \draw[crbarred,line width=.85pt,name path=crbarredchord]
          (crpt-9)--(crpt-6);
      \fi
      \foreach \crslot in {0,...,15} {
        \pgfmathsetmacro{\crturn}{180-(\crslot+.5)*360/16}
        \pgfmathtruncatemacro{\crindex}{mod(\crslot,8)/2+1}
        \pgfmathtruncatemacro{\crrole}{mod(\crslot,2)}
        \def\crlabelcolor{crunused}
        \ifnum#1=1
          \def\crlabelcolor{black}
        \else
          \ifnum\crslot=3 \def\crlabelcolor{crprimed}\fi
          \ifnum\crslot=4 \def\crlabelcolor{crprimed}\fi
          \ifnum\crslot=6 \def\crlabelcolor{crbarred}\fi
          \ifnum\crslot=9 \def\crlabelcolor{crbarred}\fi
          \ifnum\crslot=8 \def\crlabelcolor{crnonexc}\fi
          \ifnum\crslot=10 \def\crlabelcolor{crnonexc}\fi
          \ifnum\crslot=13 \def\crlabelcolor{crnonexc}\fi
          \ifnum\crslot=15 \def\crlabelcolor{crnonexc}\fi
        \fi
        \ifnum\crrole=0
          \node[rectangle,fill=\crlabelcolor,draw=\crlabelcolor,
                inner sep=0pt,minimum size=3.4pt] at (crpt-\crslot) {};
          \node[text=\crlabelcolor,fill=white,inner sep=1pt]
            at (\crturn:2.7) {$\mathrm{tgt}_{\crindex}$};
        \else
          \node[circle,fill=white,draw=\crlabelcolor,line width=.5pt,
                inner sep=0pt,minimum size=3.8pt] at (crpt-\crslot) {};
          \node[text=\crlabelcolor,fill=white,inner sep=1pt]
            at (\crturn:2.7) {$\mathrm{src}_{\crindex}$};
        \fi
      }
      \ifnum#1>6
        \path[name intersections={
          of=crfirstnon and crsecondnon,by=crcrossone}];
        \path[name intersections={
          of=crfirstnon and crbarredchord,by=crcrosstwo}];
        \node[circle,fill=black,draw=white,text=white,
              inner sep=1.5pt,font=\scriptsize] at (crcrossone) {1};
        \node[circle,fill=black,draw=white,text=white,
              inner sep=1.5pt,font=\scriptsize] at (crcrosstwo) {2};
      \fi
    \fi
    \node at (0,.86) {segment $\mathrm I$};
    \node at (0,-.57) {segment $\mathrm {II}$};
  \end{tikzpicture}%
}

\title{Bicolored Eulerian quasisymmetric functions, affine lengths, and Theta operators}
\author{Hsin-Chieh Liao}
\address{Department of Mathematics, Washington University in St. Louis,
  One Brookings Drive, St. Louis, Missouri 63130, USA}
\email{liaoh@wustl.edu}
\author{Hai Zhu}
\address{Department of Mathematics, UC San Diego, La Jolla, CA 92093, USA}
\email{haz138@ucsd.edu}
\date{\today}

\begin{document}

\begin{abstract}

We introduce a $q$-analogue of the bicolored Eulerian quasisymmetric functions by incorporating an additional grading arising from affine length, obtained by interpreting bicolored permutations as affine permutations. 
We also provide an intrinsic description of this refinement in terms of bicolored permutations by introducing notions of alignments and crossings, together with an explicit bijection identifying these statistics with the alignments and crossings in the existing literature. 
Finally, we prove that our refinement is symmetric by expressing it in terms of the Theta operators appearing in the Theta coinvariant Conjecture and the modified Macdonald symmetric functions. Assuming the Theta coinvariant conjecture implies the Schur-positivity of our refinement.

\end{abstract}

\maketitle

\section{Introduction}\label{sec:intro}

\subsection{Eulerian quasisymmetric functions}
Let $[n]\coloneqq\{1,2,\ldots,n\}$ and let $\fS_n$ denote the symmetric group, which we identify with the set of permutations of $[n]$. 
The \emph{Eulerian polynomial} can be defined as an $\fS_n$-enumerator 
\[
    A_n(t)\coloneqq \sum_{\sigma\in\fS_n} t^{\des{\sigma}}=\sum_{\sigma\in\fS_n} t^{\exc{\sigma}},
\]
where $\des{\sigma}\coloneqq\left|\{i\in [n-1] : \sigma_i>\sigma_{i+1}\}\right|$ and $\exc{\sigma}\coloneqq\left|\{i\in [n] : \sigma_i>i\}\right|$ are the descent number and the excedance number of a permutation $\sigma$, respectively. The coefficients of $A_n(t)$ are called the \emph{Eulerian numbers}.
The Eulerian numbers and polynomials have been extensively studied in the mathematics and computer science literature; see, for example, Foata and Sch\"utzenberger \cite{FoataSchutzenbergerEulerian1970} and Knuth \cite[Section 5.1]{KnuthProgramVol3}. 

In their study of the joint distribution of the major index and the excedance number in \cite{ShareshianWachs2010}, 
Shareshian and Wachs introduced a mysterious set statistic $\mathrm{DEX}: \fS_n\longrightarrow2^{[n-1]}$ and defined a family of symmetric function analogues of the Eulerian polynomial, called the \emph{Eulerian quasisymmetric functions}. In particular, one of these functions is
\[
    Q_{n}(\xxx;t)\coloneq \sum_{\sigma\in\fS_n} \FF_{\DEX{\sigma},n}(\xxx) t^{\exc{\sigma}},
\]
where $\FF_{S,n}(\xxx)$ denotes the fundamental quasisymmetric function indexed by a subset $S\subseteq [n-1]$; see Section \ref{subsec:bg-quasisymmetric} for the definition. 
The quasisymmetric function $Q_{n}(\xxx,t)$ is, as mentioned, symmetric; moreover, it is Schur-positive\footnote{In fact, $Q_{n}(\xxx,t)$ is even $h$-positive, which implies that it is the graded Frobenius series of a direct sum of graded permutation representations whose pointwise stabilizers are Young subgroups of $\fS_n$}, that is, it can be expressed as a linear combination of the Schur functions with nonnegative coefficients. This positivity has a geometric and representation-theoretic interpretation. 
The symmetric function $Q_n(\xxx,t)$ is the graded Frobenius characteristic of the $\fS_n$-representation on the cohomology ring of the permutahedral variety, studied by Procesi \cite{Procesi1990}, Stanley \cite{Stanley1989}, Stembridge \cite{Stembridge1992, Stembridge1994}, Dolgachev and Lunts \cite{dolgachev1994character} in the 90's; Equivalently, this cohomology ring can be identified with the Chow ring of the Boolean matroid, whose $\fS_n$-representation is studied by Liao \cite{Liao2023FPSAC, Liao2023One}.

Later, Hyatt extended this family by considering analogous sums of fundamental quasisymmetric functions over the wreath product $\ZZ_r\wr\fS_n$, also known as the \emph{$r$-colored permutation group}. The resulting functions, called \emph{$r$-colored Eulerian quasisymmetric functions}, are all Schur-positive, and some are even $h$-positive; see \cite{HyattColored2012} for further details. 
Despite their positivity properties, much less is known about natural $\fS_n$-representations whose graded Frobenius characteristics realize these symmetric functions.
Recently, Backman, Danner, Liao, and Stepheson \cite{BLDSFPSAC2026, LerayPerm2025} studied a bigraded algebra called the Leray model of a Boolean matroid, which extends the corresponding Chow ring. They showed that the bigraded Frobenius characteristic of the natural $\fS_n$-representation on this model is the \emph{bicolored Eulerian quasisymmetric function}, corresponding to the case $r=2$, defined by
\[
    Q_{n}^{\mathsf{lab}}(\xxx;t,y)\coloneqq \sum_{\pi\in\fB_n} \FF_{\DEX{\pi},n}(\xxx)t^{\exc{\pi}}y^{\nega{\pi}}=\sum_{k=0}^n\sum_{r=0}^{n-k} Q_{n,k,r}^{\mathsf{lab}}(\xxx)t^r y^k.
\]
where $\fB_n$ denotes the set of bicolored permutations of $[n]$. See the definitions in Section \ref{subsec:bg-bicoloredPermutations}.

In this paper, we introduce a new refinement of the bicolored Eulerian quasisymmetric function.  
We show that the refinement can be expressed in the form $\Theta_{e_k}
\Theta_{e_{\ell}}
\widetilde H_{(n-k-\ell)}\left|\right._{t=0}$ that appears in the Theta conjecture, where $\Theta_{e_k}$ is the Theta operator introduced in \cite{DAdderioIraciVandenWyngaerd2021} and $\widetilde H_{(n-k-\ell)}$ is the modified Macdonald symmetric function.
It implies that our refinement is symmetric. 
We conjecture that it is also Schur positive.

\subsection{Theta Conjecture}\label{subsec:intro-theta}

The Theta operators $\Theta_g$, indexed by a symmetric function $g$
(see Subsection~\ref{subsec:bg-theta}), were introduced by
D'Adderio, Iraci, and Vanden Wyngaerd
\cite{DAdderioIraciVandenWyngaerd2021} to refine the Delta conjectures
and study coinvariants in superspace. Their connection with the
Delta operators is expressed by the identity
$\Theta_{e_k}\nabla e_{n-k}=\Delta'_{e_{n-k-1}}e_n$
for $0\leq k<n$.

More precisely, let
\[
A_n^{(2,2)}
 :=\mathbb{C}[\mathbf{x},\mathbf{y}]
   \otimes\bigwedge(\boldsymbol{\theta},\boldsymbol{\xi}),
\qquad
R_n^{(2,2)}
 :=A_n^{(2,2)}/\big\langle
   (A_n^{(2,2)})^{\symm_n}_+\big\rangle,
\]
where each boldface alphabet consists of $n$ variables.
The bosonic variables $\mathbf{x},\mathbf{y}$ commute, whereas all fermionic variables
$\boldsymbol{\theta},\boldsymbol{\xi}$ anticommute.
The symmetric group $\symm_n$ acts by simultaneously permuting the
indices in all four alphabets, and the defining ideal $\big\langle
   (A_n^{(2,2)})^{\symm_n}_+\big\rangle$ is generated
by the invariant elements with vanishing constant term.
D'Adderio, Iraci, and Vanden Wyngaerd \cite[Conjecture 8.2]{DAdderioIraciVandenWyngaerd2021} conjectured that
\begin{equation} \label{eq:CoinvConjecture}
\operatorname{grFrob}(R_n^{(2,2)};q,t,u,v)
\stackrel{?}{=}
\sum_{\substack{k,\ell\geq0\\k+\ell<n}}
u^k v^\ell\,
\Theta_{e_k}\Theta_{e_\ell}\nabla e_{n-k-\ell},
\end{equation}
where $q,t,u,v$ record the degrees in
$\mathbf{x},\mathbf{y},\boldsymbol{\theta},\boldsymbol{\xi}$,
respectively. In particular, the term $k=\ell=0$ recovers
the Frobenius characteristic $\nabla e_n$ of diagonal coinvariants.

The same symmetric functions also arise in models of decorated
labelled Dyck paths. Let $\mathsf{LD}_n^{*k,\bullet\ell}$ denote
Dyck paths of semilength $n$ whose north steps carry positive labels
increasing up each column, with $k$ decorated rises and $\ell$
decorated contractible valleys; see \cite{InterdonatoIraci2026} for details. Interdonato and Iraci
\cite{InterdonatoIraci2026} proposed the bivariate Theta conjecture
\[
\Theta_{e_k}\Theta_{e_\ell}\nabla e_{n-k-\ell}
\stackrel{?}{=}
\sum_{P\in\mathsf{LD}_n^{*k,\bullet\ell}}
q^{\operatorname{dinv}(P)}
t^{\operatorname{area}(P)}x^P,
\qquad k+\ell<n.
\]
Setting $\ell=0$ or $k=0$ recovers the rise or valley version of the
Delta conjecture, respectively. Its $q=1$ specialization was
conjectured in \cite{DAdderioIraciVandenWyngaerd2021} and proved in
\cite{DAdderioInterdonatoIraciPagaria2026}.

Setting $\mathbf{y}=0$ gives the $(1,2)$-bosonic-fermionic
coinvariant ring
$R_n^{(1,2)}
=(\mathbb{C}[\mathbf{x}]\otimes
\bigwedge(\boldsymbol{\theta},\boldsymbol{\xi}))
/\langle(\mathbb{C}[\mathbf{x}]\otimes
\bigwedge(\boldsymbol{\theta},\boldsymbol{\xi}))^{S_n}_+\rangle$.
Accordingly, the $t=0$ specialization predicts that its component
of fermionic bidegree $(k,\ell)$ has graded Frobenius characteristic
\[
\operatorname{SF}(n,k,\ell)
:=
\left.
\Theta_{e_k}\Theta_{e_\ell}\nabla e_{n-k-\ell}
\right|_{t=0}
=
\left.
\Theta_{e_k}\Theta_{e_\ell}
\widetilde H_{(n-k-\ell)}
\right|_{t=0}.
\]
Independent of this representation-theoretic conjecture,
Iraci, Nadeau, and Vanden Wyngaerd \cite{MR4766722} proved the combinatorial formula
\[
\operatorname{SF}(n,k,\ell)
=
\sum_{w\in\mathsf{SW}(n,k,\ell)}
q^{\operatorname{sminv}(w)}x^w,
\]
where $\mathsf{SW}(n,k,\ell)$ denotes the set of segmented Smirnov words of length $n$ with $k$ ascents and $\ell$ descents. 
See Equation~\eqref{eq:theta-known-smirnov-expansion} for its fundamental quasisymmetric function expansion.
Lentfer \cite{Lentfer2025} subsequently proposed a monomial basis
using modified Motzkin paths and established a weight-preserving
bijection with segmented permutations, showing that the resulting
candidate Frobenius series agrees with the Smirnov-word formula.
His conjectural basis and bijection play crucial roles in Section~\ref{sec:proofs-theta-formulas}.

\subsection{Main Results}
We start by introducing our new refinement of the bicolored Eulerian quasisymmetric functions.
For background on bicolored permutations and affine permutations, we refer the reader to Section \ref{subsec:bg-bicoloredPermutations} and \ref{subsec:bg-affine} for further details.

For each bicolored permutation $\pi\in\fB_n$, we associate an affine permutation $f_{\pi}$.
We write $\ell(f_\pi)$ for the \emph{affine length} of $f_{\pi}$.
Let $\coloredperm{n}{k}$ denote the subset of $\fB_n$ consisting of bicolored permutations $\pi$ with $\nega{\pi}=k$. 

\begin{definition}\label{def:al}
Let $\pi\in\coloredperm{n}{k}$ with $\exc{\pi}=r$. The \emph{alignment number} of $\pi$, denoted by $\al{\pi}$, is defined by normalizing the affine length of $f_{\pi}$ as follows:
\begin{equation}\label{eq:al-from-affine-length}
\al{\pi}:=\ell(f_\pi)-r(n-k-r).
\end{equation}
\end{definition}

We then define our refinement of the bicolored Eulerian quasisymmetric function as follows.

\begin{definition}\label{def:bicolored-eulerian-qsf}
For $n\geq1$, $0\leq k\leq n$, and $0\leq r\leq n-k$, define the
\emph{$q$-analogue of the bicolored Eulerian quasisymmetric function}
by
\begin{equation}\label{eq:bicolored-eulerian-qsf}
\coloredfunc{n}{k}{r}(\xxx;q)
\coloneqq
\sum_{\substack{\pi\in\coloredperm{n}{k}\\
                  \exc{\pi}=r}}
q^{\al{\pi}}
\FF_{\DEX{\pi},n}(\xxx).
\end{equation}
Specializing at $q=1$ recovers the bicolored Eulerian
quasisymmetric function $Q_{n,k,r}^{\mathsf{lab}}(\xxx)$.
\end{definition}

\begin{remark}
The normalization term $r(n-k-r)$ in \eqref{eq:al-from-affine-length} is chosen so that the lowest power of $q$ with a nonzero coefficient in $\coloredfunc{n}{k}{r}(\xxx;q)$ is $q^0$, provided that the indexing set is nonempty.
Indeed, Corollary~\ref{cor:minimum-affine-length} shows that, under $n\geq1$, $0\leq k\leq n$, and $0\leq r\leq n-k$, $\al{\pi}\ge 0$ for every $\pi\in\coloredperm{n}{k}$ with $\exc{\pi}=r$, and the minimum value of $\al{\pi}$ over this set is zero whenever the set is nonempty. The indexing set $\{\pi\in \coloredperm{n}{k} : \exc{\pi}=r\}=\varnothing$ if and only if $(k,r)=(0,n)$.
\end{remark}

In Section \ref{sec:crossings-and-affine-length}, we give a combinatorial formulation of $\al{\pi}$ independent of the associated affine permutation $f_{\pi}$, which justifies its name ``alignment number'', together with the notion of crossing numbers of $\pi$.  


\begin{definition}\label{def:one-bar-status-parts}
For the same $n,k,r$ as in Definition \ref{def:bicolored-eulerian-qsf}, define the \emph{unbarred-one part} by
\begin{equation}\label{eq:unbarred-one-part}
\Phi^{\mathrm{unb}}_{n,k,r}(\xxx;q)
:=
\sum_{\substack{\pi\in\coloredperm{n}{k}\\
                  \exc{\pi}=r\\
                  \text{$1$ is unbarred in $\pi$}}}
q^{\al{\pi}}
\FF_{\DEX{\pi},n}(\xxx),
\end{equation}
and the \emph{barred-one part} by
\begin{equation}\label{eq:barred-one-part}
\Phi^{\mathrm{bar}}_{n,k,r}(\xxx;q)
:=
\sum_{\substack{\pi\in\coloredperm{n}{k}\\
                  \exc{\pi}=r\\
                  \text{$1$ is barred in $\pi$}}}
q^{\al{\pi}}
\FF_{\DEX{\pi},n}(\xxx).
\end{equation}
\end{definition}

Note that if $n-k-r=0$, then the sum in
\eqref{eq:unbarred-one-part} is empty; whereas if $k=0$, then the sum in
\eqref{eq:barred-one-part} is empty.
Also, we have
\begin{equation}\label{eq:one-bar-status-decomposition}
\coloredfunc{n}{k}{r}(\xxx;q)
=
\Phi^{\mathrm{unb}}_{n,k,r}(\xxx;q)
+
\Phi^{\mathrm{bar}}_{n,k,r}(\xxx;q).
\end{equation}


We next state the symmetric-function formulas for the two
functions above, related to the Theta conjecture. For the remainder of the introduction, set
\begin{equation}\label{eq:intro-u-delta}
u\coloneqq n-k-r,
\qquad
\delta_{n,k,r}
\coloneqq
\binom{n}{2}-\binom{r+1}{2}-\binom{u}{2}.
\end{equation}

\begin{theorem}[Reduced unbarred-one Theta formula]
\label{conj:unbarred-one-theta}
Let $n\geq1$, $0\leq k< n$, and $0\leq r< n-k$, and hence
$u\geq1$.  Then
\begin{equation}\label{eq:unbarred-one-theta-conjecture}
\Phi^{\mathrm{unb}}_{n,k,r}(\xxx;q)
=
q^{\delta_{n,k,r}}
\omega\!\left(
\left.
\big(\Theta_{e_r}
\Theta_{e_{u-1}}
\widetilde H_{(k+1)}
\big)\right|_{t=0,\,q=q^{-1}}\right),
\end{equation}
where the $(q,t)$-specialization is defined by \eqref{eq:bg-grading-specialization}.
\end{theorem}

\begin{remark}
The right-hand side is the reversal of $\omega\!\left(
\left.
\big(\Theta_{e_r}
\Theta_{e_{u-1}}
\widetilde H_{(k+1)}
\big)\right|_{t=0}\right)$ as a polynomial in $q$. We apply $\omega$ and this reversal to make the coefficient of $q^0$ equal to $h_n$, which looks more like the graded Frobenius character of a quotient ring.   
\end{remark}



The corresponding non-reduced formula, without imposing a condition
on the bar status of the value $1$, is as follows. Here and throughout, we adopt the convention
$e_j=0$ for $j<0$, so that
$\Theta_{e_{-1}}=0$.

\begin{theorem}[Non-reduced bicolored Theta formula]
\label{conj:bicolored-theta}
Let $n\geq1$, $0\leq k\leq n$, and $0\leq r\leq n-k$.  Then
\begin{align}
\coloredfunc{n}{k}{r}(\xxx;q)
={}&
q^{\delta_{n,k,r}}
\omega\!\left(
\left.\big(
\Theta_{e_r}
\Theta_{e_{u-1}}
\widetilde H_{(k+1)}
\big)\right|_{t=0,\,q=q^{-1}}\right)
\notag\\
&+
q^{\delta_{n,k,r}}
\omega\!\left(
\left.\big(
\Theta_{e_r}
\Theta_{e_u}
\widetilde H_{(k)}\big)
\right|_{t=0,\,q=q^{-1}}
\right).
\label{eq:bicolored-theta-conjecture}
\end{align}
When $u=0$, the first summand on the right-hand side of
\eqref{eq:bicolored-theta-conjecture} vanishes; whereas when $k=0$, the
second summand vanishes. 
\end{theorem}

Proposition~\ref{prop:barred-unbarred-relation} below will show that
the barred-one part with $k$ bars is a grading shift of the
unbarred-one part with $k-1$ bars.  Consequently, when $k\geq1$ and
$u\geq1$, Theorem~\ref{conj:bicolored-theta} follows formally from
the two instances of Theorem~\ref{conj:unbarred-one-theta} with bar
parameters $k$ and $k-1$.  If $u=0$, only the instance with parameter
$k-1$ is needed; if $k=0$ and $u\geq1$, only the instance with
parameter $k$ is needed.  The remaining case $k=u=0$ is the empty
case.
Indeed, the reduced formula with bar parameter $k-1$ has
$u+1$ nonexcedance values, and
\[
\delta_{n,k-1,r}=\delta_{n,k,r}-u.
\]
Thus the factor $q^u$ in
\eqref{eq:barred-unbarred-function-relation} is absorbed by the change
in the exponent of $q^{\delta}$, which is why the second summand in
\eqref{eq:bicolored-theta-conjecture} has no additional power of $q$.

Theorems~\ref{conj:unbarred-one-theta} and \ref{conj:bicolored-theta} imply the following nontrivial result.

\begin{corollary}
The quasisymmetric functions $\Phi^{\mathrm{unb}}_{n,k,r}(\xxx;q)$, $\Phi^{\mathrm{bar}}_{n,k,r}(\xxx;q)$, and  $\coloredfunc{n}{k}{r}(\xxx;q)$ are symmetric.    
\end{corollary}



The following conjecture is true assuming the Theta coinvariant Conjecture \ref{eq:CoinvConjecture} is true. 

\begin{conjecture} \label{conj:SchurPositive}
For $n\ge 1$, $0\le k\le n$, and $0\le r\le n-k$, the symmetric function $\Phi^{\mathrm{unb}}_{n,k,r}(\xxx;q)$ is Schur-positive. 
Consequently, so are $\Phi^{\mathrm{bar}}_{n,k,r}(\xxx;q)$ and $\Phi_{n,k,r}(\xxx;q)$
\end{conjecture}

\begin{remark}\label{rem:degree-normalization-in-conjectures}
The power of $q$ and the inversion of the Macdonald parameter in
\eqref{eq:unbarred-one-theta-conjecture} are consequences of the
normalization in Definition~\ref{def:al}.  Indeed,
Corollary~\ref{cor:cr-al-complement-for-pi} gives
\[
\al{\pi}
=
\delta_{n,k,r}-\cro{\pi}.
\]
Thus a crossing generating function in a variable $z$ becomes the
present generating function after the substitutions
\[
z=q^{-1}
\qquad\text{and}\qquad
F(\xxx;z)\longmapsto q^{\delta_{n,k,r}}F(\xxx;q^{-1}).
\]
Since $\coloredfunc{n}{k}{r}$ is already defined using the normalized
length $\ell(f_\pi)-r(n-k-r)$, no additional factor
$q^{-r(n-k-r)}$ occurs in either theorem.
\end{remark}

The rest of this paper is organized as follows.
Section~\ref{sec:background} reviews the necessary background.
In Section~\ref{sec:crossings-and-affine-length}, we introduce the notions of a crossing and an alignment of a bicolored permutation and relate them to affine Coxeter length. 
We then establish the relationship between the barred-one part $\Phi^{\mathrm{bar}}_{n,k,r}$ and the
unbarred-one part $\Phi^{\mathrm{unb}}_{n,k,r}$. In Section~\ref{sec:proofs-theta-formulas}, we prove the reduced
and non-reduced Theta formulas using a bijection with segmented
permutations.
In Section~\ref{sec:comparison-with-signed-statistics}, we relate
our crossings and alignments on bicolored permutations with the crossings defined by Corteel, Josuat-Verg\'es, Kim \cite{CorteelJosuatVergesKim2013} and the alignments defined by Cho and Park \cite{ChoPark2015}, with an explicit bijection on $\fB_n$ that translates the corresponding statistics.
Finally, Section~\ref{sec:future-directions} proposes the problem
of constructing natural graded algebra models for our
symmetric functions.

\section{Background}\label{sec:background}
\subsection{Bicolored permutations}\label{subsec:bg-bicoloredPermutations}

Denote $[\bar{n}]\coloneqq\{\overline{i} : i\in [n]\}$.
A \emph{bicolored permutation} $\pi$ on $[n]$ is a bijection on $[n]\cup[\bar{n}]$ satisfying $\pi(\overline{i})=\overline{\pi(i)}$ for any $i\in [n]\cup[\overline{n}]$.
For simplicity, we write $\pi_i$ as the image $\pi(i)$ and  adopt the convention that $\overline{\overline{i}}=i$. 
Since a bicolored permutation $\pi$ is solely determined by its image on $[n]$, we may write $\pi=\pi_1\pi_2\cdots\pi_n$, which is called the \emph{window notation} of $\pi$.

We say that a bicolored permutation $\pi$ has $k$ bars, denoted by  $\numcolor{\pi}=k$, if there are $k$ barred entries in its window notation.
For $i\in [n]$, we write $|\bar{i}|\coloneqq i$ and $|\pi|$ to be the permutation obtained from $\pi$ by removing all the bars. 
Let $\fB_n$ be the set of bicolored permutations on $[n]$ and $\coloredperm{n}{k}$ be the subset of those permutations with $k$ bars. 
Most of the time throughout this paper, we adopt the following total order on the bicolored integers from $1$ to $n$:
\begin{equation} \label{eq:ColorOrder}
    \overline{1}<\overline{2}<\cdots<\overline{n}<1<2<\cdots<n
\end{equation}



Under this order, we define the \emph{excedance set} $\EXC{\pi}$ and the \emph{excedance number} $\exc{\pi}$ of a bicolored permutation $\pi$ by
\[
\EXC{\pi}:=\{i\in[n]:\pi_i>i\},
\qquad
\exc{\pi}:=|\EXC{\pi}|,
\]

\begin{definition}\label{def:DEX}
For $\pi\in\fB_n$, we construct a word 
$\pi^D=\pi_1^D\pi_2^D\cdots\pi_n^D$ obtained from $\pi$ defined by
\[
    (\pi_i^D)\coloneq\begin{cases}
        \pi_i^\prime & \text{ if }i\in \EXC{\pi}\\
        \pi_i   & \text{otherwise}.
    \end{cases}
\]
So $\pi^D$ is a word over the alphabet set $[n]\cup [\bar{n}]\cup[n^\prime]$ where $[n^\prime]=\{i^\prime : i\in [n]\}$.
Extend the total order \eqref{eq:ColorOrder} to
\[
1'<2'<\cdots<n'
<\overline{1}<\overline{2}<\cdots<\overline{n}
<1<2<\cdots<n.
\]
and define, under this order, the set statistic
\[
\DEX{\pi}:=\DES{\pi^D}=\{i\in [n-1]: \pi_i^D>\pi_{i+1}^D\}.
\]
\end{definition}
\begin{remark}
The set of permutations $\fS_n$ can be viewed as a subset of $\fB_n$. 
Restrict Definition \ref{def:DEX} from $\fB_n$ to $\fS_n$ recovers the original definition of $\mathrm{DEX}$ by Shareshian and Wachs in \cite{ShareshianWachs2010}. 
\end{remark}

A letter $\pi_i^D$ is called \emph{barred}, \emph{primed}, or
\emph{ordinary} according as
\[
\pi_i^D=\overline{\sigma_i},\qquad
\pi_i^D=\sigma_i^\prime,\qquad\text{or}\qquad
\pi_i^D=\sigma_i,
\]
respectively. By a slight abuse of terminology, we also say that the
value $\sigma_i$ is barred, primed, or ordinary, respectively.

\begin{remark}
Another common convention for bicolored permutations is to interpret the bar as a negative sign, identifying $\overline{i}$ with $-i$ for each $i\in[n]$. Under this convention, we regard $[n]\cup[\overline{n}]$ as a subset of $\ZZ$ equipped with the total order
\begin{equation} \label{eq:SignOrder}
    \overline{n}<\cdots<\overline{2}<\overline{1}<1<2<\ldots<n.
\end{equation}
We may call this the \emph{sign order}, and call \eqref{eq:ColorOrder} the \emph{color order}.
It seems to us that the color order is more natural when regarding the hyperoctahedral group $\fB_n$ as a wreath product, whereas the sign order is more natural when regarding $\fB_n$ as a Coxeter group. 
\end{remark}



\subsection{Affine permutations and Coxeter length}
\label{subsec:bg-affine}

Fix $n\geq1$. An \emph{affine permutation} of period $n$
is a bijection $f:\mathbb Z\to\mathbb Z$ satisfying
\[
f(i+n)=f(i)+n\qquad(i\in\mathbb Z).
\]
Such a bijection is uniquely determined by its values on $[n]$, so we write $f=[f(1),\ldots,f(n)]$, called the \emph{window notation} of $f$. Conversely,
an integer window defines an affine permutation of period $n$ if and only if its entries represent all congruence classes modulo $n$. 

The \emph{affine symmetric group} $\dschub_n$ is the group of all affine permutations of period $n$, with composition as the group operation. 
The \emph{displacement map} $\operatorname{disp}:\dschub_n\rightarrow \ZZ$, defined by
\begin{equation}\label{eq:bg-affine-displacement}
\operatorname{disp}(f)
:=\frac1n\sum_{i=1}^{n}(f(i)-i),
\end{equation}
is a group homomorphism. 
For each $a\in\ZZ$, let 
$\dschub_n^a=\{f\in\dschub_n:\operatorname{disp}(f)=a\}$.
In particular, $\widetilde{S}_n^0=\ker(\operatorname{disp})$ is a normal subgroup of $\widetilde{S}_n$.
Moreover, $\widetilde{S}_n^a=\zeta_a \widetilde{S}_n^0$, where $\zeta_a\in\dschub_n$ is the translation defined by $\zeta_a(i)=i+a$ for all $a\in\ZZ$.
Thus each $\widetilde{S}_n^a$ is a coset of $\widetilde{S}_n^0$, .

For $n\geq2$, the subgroup $\dschub_n^0$ is the affine
Coxeter group of type $\widetilde A_{n-1}$, generated by the
periodic adjacent transpositions: for each $i\in\{0,\ldots,n-1\}$,
the corresponding generator interchanges $i+bn$ and $i+1+bn$ for
every $b\in\mathbb Z$, fixing all other integers.
The \emph{Coxeter length} of an element $f\in\dschub_n^0$, denoted by $\ell(f)$, is the
minimum number of these generators required to express $f$. See \cite[Section~3.2]{KnutsonLamSpeyer2013} and
\cite[Section~6.2]{speyer2024richardsonvarietiesprojectedrichardson}.

Since $\operatorname{disp}(\zeta_a\circ f)=a+\operatorname{disp}(f)$ for all $a\in \ZZ$ and $f\in\dschub_n$, 
we define the \emph{affine length} of $f\in\dschub_n$, for $n\ge 2$, by
\[
\ell(f)
:=\ell\bigl(\zeta_{-\operatorname{disp}(f)}\circ f\bigr),
\]
where the length on the right-hand side is the Coxeter length in $\dschub_n^0$.
For $n=1$ all elements are global shifts; define their length to
be zero.
Equivalently, $\ell(f)$ is the number of equivalence classes of
pairs $(i,j)$ satisfying $i<j$ and $f(i)>f(j)$, under simultaneous
translation of both indices by $n$. Each class has a unique
representative whose first index is in $[n]$, so
\begin{equation}\label{eq:bg-affine-inversions}
\ell(f)
=\#\{(i,j)\in[n]\times\ZZ\,:\,i<j,\ f(i)>f(j)\}.
\end{equation}
Note that here $j$ is not restricted to $[n]$. This set is finite because
$f(i)-i$ is periodic and hence bounded, see \cite[Section~3.2]{KnutsonLamSpeyer2013}.
An alternative formula for the affine length is
\begin{equation}\label{eq:bg-affine-window-length}
\ell(f)
=\sum_{1\leq i<j\leq n}
\left|\left\lfloor\frac{f(j)-f(i)}{n}\right\rfloor\right|,
\end{equation}
see \cite[Proposition 3.1]{MR1392503}.

One can also embed $\fS_n$ in $\dschub_n$ by extending an ordinary permutation $w\in\symm_n$ periodically, then 
\eqref{eq:bg-affine-window-length} is equal to the inversion number of $w$, $\mathrm{inv}(w)\coloneqq\#\{(i,j)\in[n]^2: i<j,~w_i>w_j\}$, agrees with the usual Coxeter length for $\fS_n$.
Adding the same integer to all window entries leaves the length
unchanged, so $\ell(\zeta_a\circ f)=\ell(f)$ and
$\ell(\zeta_a)=0$. When $n=1$, this agrees with the empty sum in
\eqref{eq:bg-affine-window-length}.

\subsection{Combinatorics}\label{subsec:bg-combinatorics}

For an integer $n\geq0$, write $[n]=\{1,\ldots,n\}$, with
$[0]=\varnothing$. A \emph{partition} of $n$ is a weakly decreasing sequence
$\lambda=(\lambda_1,\ldots,\lambda_d)$ of positive integers with
$|\lambda|:=\lambda_1+\cdots+\lambda_d=n$; we write $\lambda\vdash n$.
The unique partition of zero is the empty partition $\varnothing$.
We use \emph{English notation} for the \emph{Young diagram}
\[
\{(i,j):1\leq i\leq d,\ 1\leq j\leq\lambda_i\}
\]
of $\lambda$, whose rows (the first coordinate $i$) are numbered from top to bottom and columns (the second coordinate $j$) from left to
right. The \emph{conjugate partition} is given by
$\lambda'_j=|\{i:\lambda_i\geq j\}|$. For partitions $\lambda,\mu$ of
the same integer, write $\mu\unrhd\lambda$ if
\[
\sum_{i=1}^{a}\mu_i\geq\sum_{i=1}^{a}\lambda_i
\qquad\text{for every }a\geq1,
\]
where missing parts are taken to be zero (i.e. for $\lambda=(\lambda_1,\dots,\lambda_d)$ and $i>d$ we have $\lambda_i=0$). This is the \emph{dominance order}.

A \emph{semistandard Young tableau} of shape $\lambda$ is a filling
of its diagram by positive integers that weakly increase along rows
and strictly increase down columns. A \emph{standard Young tableau}
of shape $\lambda\vdash n$ is a semistandard Young tableau containing each element of $[n]$ exactly
once, so these elements strictly increase along both rows and columns. Denote
these sets by $\SSYT(\lambda)$ and $\SYT(\lambda)$, respectively.
For $\lambda\vdash n$ with $n\geq1$ and $T\in\SYT(\lambda)$, define
\[
\operatorname{Des}(T)
:=\{i\in[n-1]:\text{$i+1$ lies in a lower row than $i$ in $T$}\}.
\]
Transposing the diagram and its filling gives a tableau
$T^{\mathsf t}\in\SYT(\lambda')$.
For the unique empty tableau, put $\operatorname{Des}(T)=\varnothing$.

Let $w=w_1\cdots w_n$ be a word with pairwise distinct letters in a totally ordered alphabet.
Its \emph{inversion number} is
\begin{equation}\label{eq:bg-word-inversions}
\inv(w):=
\bigl|\{(i,j):1\leq i<j\leq n,\ w_i>w_j\}\bigr|.
\end{equation}
Its
\emph{standardization} is the permutation
$\operatorname{std}(w)\in\symm_n$ obtained by labeling its letters
with $1,\ldots,n$ in increasing order. Explicitly,
\[
\operatorname{std}(w)_i
=\bigl|\{j\in[n]:w_j\le w_i\}\bigr|.
\]
In particular, standardization preserves inversions and descents: we have
$\inv(\operatorname{std}(w))=\inv(w)$ and
$\DES{\operatorname{std}(w)}=\DES{w}$. For a permutation $w\in\symm_n$, we have $\inv(w)=\inv(w^{-1})$: the map
$(i,j)\mapsto(w_j,w_i)$ bijects the inversion pairs of $w$ with
those of $w^{-1}$.

\subsection{Quasisymmetric functions}\label{subsec:bg-quasisymmetric}

Let $\xxx=(x_1,x_2,\ldots)$ be commuting indeterminates. A formal
power series of bounded total degree over the ground field $\mathbb{Q}(q,t)$ is
\emph{quasisymmetric} if, for every sequence
$\alpha=(\alpha_1,\ldots,\alpha_d)$ of positive integers, the
coefficient of
$x_{i_1}^{\alpha_1}\cdots x_{i_d}^{\alpha_d}$ is independent of the
choice of $i_1<\cdots<i_d$. Write
\[
\operatorname{QSym}
=\bigoplus_{n\geq0}\operatorname{QSym}_n
\]
for the graded algebra of quasisymmetric functions, where
$\operatorname{QSym}_n$ is its homogeneous degree-$n$ component.

For $n\geq1$ and $S\subseteq[n-1]$, the \emph{fundamental quasisymmetric
function}, in the convention of this paper, is
\begin{equation}\label{eq:bg-fundamental}
\FF_{S,n}(\xxx)
:=\sum_{\substack{1\leq i_1\leq\cdots\leq i_n\\
                   a\in S\ \Longrightarrow\ i_a<i_{a+1}}}
x_{i_1}\cdots x_{i_n}.
\end{equation}
The functions $\FF_{S,n}$, as $S$ ranges over all subsets of
$[n-1]$, form a basis of $\operatorname{QSym}_n$; see
\cite{Gessel1984}. In degree zero we use the basis element $1$. 
We use increasing indices; the decreasing-index convention of
\cite{ShareshianWachs2010,LerayPerm2025} instead defines
\[
\FF^{\downarrow}_{S,n}(\xxx)=
\sum_{\substack{i_1\geq\cdots\geq i_n\geq1\\
j\in S\Rightarrow i_j>i_{j+1}}}
x_{i_1}\cdots x_{i_n}.
\]
Reversing the index sequence gives
$\FF_{S,n}=\FF^{\downarrow}_{n-S,n}$, where
$n-S\coloneqq\{n-j:j\in S\}$. This convention difference does not matter as both $Q_n(\xxx;t)$ and $Q_{n}^{\mathsf{lab}}(\xxx;t,y)$ are symmetric functions; see the next paragraph for details.

A quasisymmetric function is \emph{symmetric} if it is invariant
under every finite permutation of the variables $\xxx$. For a homogeneous symmetric function $f$ of degree $n$, the two
conventions give the same fundamental expansion coefficients:
\[
f=\sum_{S\subseteq[n-1]}c_S\FF_{S,n}
\quad\Longleftrightarrow\quad
f=\sum_{S\subseteq[n-1]}c_S\FF^{\downarrow}_{S,n}.
\]
Indeed, after specializing to $x_1,x_2,\dots,x_N$, reversing the variables
fixes every symmetric polynomial and interchanges $\FF_{S,n}$ with
$\FF^{\downarrow}_{S,n}$. The assertion follows since this holds for
every $N$.

\subsection{Symmetric functions}\label{subsec:bg-symmetric}

A quasisymmetric function is \emph{symmetric} if it is invariant
under every permutation of finitely many variables in $\xxx$. Let
$\Lambda\subseteq\operatorname{QSym}$ be the subalgebra of symmetric
functions, with homogeneous components $\Lambda_n$.
For $d\geq1$, the \emph{elementary, complete homogeneous, and power-sum
symmetric functions} are, respectively,
\begin{align*}
e_d(\xxx):=\sum_{i_1<\cdots<i_d}x_{i_1}\cdots x_{i_d},\quad
h_d(\xxx):=\sum_{i_1\leq\cdots\leq i_d}x_{i_1}\cdots x_{i_d},\quad
p_d(\xxx):=\sum_{i\geq1}x_i^d.
\end{align*}
Here the indices in the first two sums are positive, and
$e_0=h_0=1$. These three families of symmetric functions are all algebraically independent, respectively, and we have $\Lambda=\mathbb Q(q,t)[p_1,p_2,\ldots]=\QQ(q,t)[e_1,e_2,\ldots]=\QQ(q,t)[h_1,h_2,\ldots]$.

For a partition $\lambda$, its \emph{Schur function} is
\begin{equation}\label{eq:bg-schur-tableaux}
s_\lambda(\xxx)
:=\sum_{T\in\SSYT(\lambda)}
\prod_{(i,j)\in\lambda}x_{T(i,j)}.
\end{equation}
The Schur functions indexed by partitions of $n$ form a basis of
$\Lambda_n$. We use the \emph{Hall scalar product}, characterized by
$\langle s_\lambda,s_\mu\rangle=\delta_{\lambda,\mu}$, and extend it
bilinearly over any coefficient field in use. Schur functions have the fundamental quasisymmetric function expansion
\begin{equation}\label{eq:bg-schur-fundamental}
s_\lambda(\xxx)
=\sum_{T\in\SYT(\lambda)}
\FF_{\operatorname{Des}(T),n}(\xxx)
\qquad(\lambda\vdash n,\ n\geq1).
\end{equation}
These conventions are standard; see \cite{Macdonald1995,Gessel1984}.

\subsection{The involutions
\texorpdfstring{$\Omega_n$ and $\omega$}{Omega n and omega}}
\label{subsec:bg-involutions}

The standard involution $\omega$ of $\Lambda$ is the algebra
automorphism satisfying $\omega(e_d)=h_d$. Equivalently,
\[
\omega(p_d)=(-1)^{d-1}p_d,
\qquad
\omega(s_\lambda)=s_{\lambda'}.
\]
It fixes the scalar coefficients, including all grading and
parameters $q,t$.

For $n\geq1$, define the linear involution $\Omega_n$ of
$\operatorname{QSym}_n$ by
\begin{equation}\label{eq:bg-omega-qsim}
\Omega_n(\FF_{S,n})
:=\FF_{[n-1]\setminus S,n},
\end{equation}
and put $\Omega_0(1)=1$. Its restriction to $\Lambda_n$ equals
$\omega$. To verify this, observe that
\begin{equation}\label{eq:bg-tableau-transpose-descents}
\operatorname{Des}(T^{\mathsf t})
=[n-1]\setminus\operatorname{Des}(T)
\end{equation}
for all standard Young tableaux $T$.
Indeed, if the entry $i+1$ were strictly southeast of $i$, the cell
in the row of $i$ and the column of $i+1$ would contain an integer
strictly between $i$ and $i+1$, which is impossible. The increasing conditions also rule out
$i+1$ being weakly northwest of $i$. Consequently, exactly one of
the following holds: $i+1$ is in a lower row than $i$, or it is in
a column strictly to the right of $i$. Transposition exchanges
these two conditions, proving
\eqref{eq:bg-tableau-transpose-descents}. Thus
\eqref{eq:bg-schur-fundamental} yields
\[
\Omega_n(s_\lambda)
=\sum_{T\in\SYT(\lambda)}
  \FF_{\operatorname{Des}(T^{\mathsf t}),n}
=s_{\lambda'}
=\omega(s_\lambda).
\]
Since the Schur functions form a linear basis of $\Lambda$, the claimed restriction
follows. In particular, if
$G=\sum_{S\subseteq[n-1]}c_S\FF_{S,n}$ is symmetric, then
\begin{equation}\label{eq:bg-complement-symmetric}
\omega(G)
=\sum_{S\subseteq[n-1]}c_S
  \FF_{[n-1]\setminus S,n}.
\end{equation}
For a general quasisymmetric $G$, the right-hand side instead
defines $\Omega_n(G)$; the identification with $\omega$ requires
symmetry.

\subsection{Modified Macdonald symmetric functions and Theta operators}\label{subsec:bg-theta}

For symmetric functions $G,A=A(\xxx;q,t)\in\Lambda$, the \emph{plethystic substitution} $G[A]$ is defined by expressing $G$ in
the algebraically independent power sums and replacing each $p_d$
by $p_d[A]\coloneqq A(x_1^d,x_2^d,\ldots;q^d,t^d)$, while leaving its coefficients unchanged. In particular, the
substitutions needed here are determined, for $d\geq1$, by
\begin{align*}
p_d[(1-q)\xxx]&=(1-q^d)p_d(\xxx),\\
p_d[(1-t)\xxx]&=(1-t^d)p_d(\xxx),\\
p_d\!\left[\frac{\xxx}{(1-q)(1-t)}\right]
&=\frac{p_d(\xxx)}{(1-q^d)(1-t^d)},
\end{align*}
where we abuse notation by writing $\xxx=\sum_{i\ge 1}x_i$.

For $\lambda\vdash n\geq1$, the \emph{modified Macdonald symmetric
function} $\widetilde H_\lambda=\widetilde H_\lambda(\xxx;q,t)$ is
uniquely characterized by
\begin{align}
\widetilde H_\lambda[(1-q)\xxx]
&\in\operatorname{span}_{\mathbb Q(q,t)}
  \{s_\mu:\mu\vdash n,\ \mu\unrhd\lambda\},\notag\\
\widetilde H_\lambda[(1-t)\xxx]
&\in\operatorname{span}_{\mathbb Q(q,t)}
  \{s_\mu:\mu\vdash n,\ \mu\unrhd\lambda'\},\notag\\
\langle\widetilde H_\lambda,s_{(n)}\rangle&=1.
\label{eq:bg-modified-macdonald}
\end{align}
Set $\widetilde H_\varnothing=1$. These functions form a basis of
$\Lambda$; see \cite{HaglundHaimanLoehr2005}.

Define an invertible operator $\Pi$ on the positive-degree part of
$\Lambda$ by
\begin{equation}\label{eq:bg-pi-operator}
\Pi\widetilde H_\lambda
:=\left(
\prod_{\substack{(i,j)\in\lambda\\(i,j)\neq(1,1)}}
(1-q^{j-1}t^{i-1})
\right)\widetilde H_\lambda
\qquad(\lambda\neq\varnothing).
\end{equation}
The eigenvalues in \eqref{eq:bg-pi-operator} are nonzero in
$\mathbb Q(q,t)$.

For homogeneous $f\in\Lambda$ of degree $d$ and homogeneous
$G\in\Lambda$ of degree $n$, define
\begin{equation}\label{eq:bg-theta-definition}
\Theta_fG:=
\begin{cases}
0,&n=0,\ d>0,\\[2pt]
fG,&n=d=0,\\[2pt]
\displaystyle
\Pi\!\left(
 f\!\left[\frac{\xxx}{(1-q)(1-t)}\right]\Pi^{-1}G
\right),&n>0,
\end{cases}
\end{equation}
and extend this definition linearly in both arguments $f,G\in\Lambda$. This is the
\emph{Theta operator} of \cite[Section~3]{DAdderioIraciVandenWyngaerd2021},
including its degree-zero convention. It raises degree by $d$,
$\Theta_{e_0}=\operatorname{id}$, and $\Theta_f(1)=0$ when $d>0$.
If $G$ has positive degree, the conjugation formula gives commutativity of Theta operators
\begin{align}\label{eq:theta-operator-commutativity}
\Theta_f\Theta_gG
=\Pi\!\left(
 f\!\left[\frac{\xxx}{(1-q)(1-t)}\right]
 g\!\left[\frac{\xxx}{(1-q)(1-t)}\right]\Pi^{-1}G
\right)
=\Theta_g\Theta_fG.
\end{align}

Specialization is coefficientwise, \emph{after} forming the complete
Theta expression.
To distinguish the Macdonald parameter from the external grading
variable, for a fresh indeterminate $z$ set
\begin{equation}\label{eq:bg-grading-specialization}
\left.G\right|_{t=0,\,q=q^{-1}}
:=
\left.\left(\left.G\right|_{t=0,\,q=z}\right)\right|_{z=q^{-1}}.
\end{equation}

\section{Crossings and normalized affine length}
\label{sec:crossings-and-affine-length}

In this section, we define some statistics and show some identities about them.

We start by defining a way to associate each bicolored permutation $\pi\in\fB_n$ with an affine permutation $f_\pi\in\dschub_n$.

For each $\pi\in\fB_n$, we define the corresponding affine permutation $f_\pi$ by setting its window notation $[f_\pi(1),f_\pi(2),\ldots,f_\pi(n)]$ to be
\[
    f_{\pi}(i)\coloneqq\begin{cases}
        \pi_i-2n & \text{ if }\pi_i>i\\
        |\pi_i|-n & \text{ if $\pi_i$ is barred}\\
        \pi_i   & \text{ otherwise}
    \end{cases}.
\]
Equivalently, the window notation can be obtained from the word $\pi^D$ in Definition~\ref{def:DEX} by making the
replacements
\begin{equation}\label{eq:intro-affine-replacements}
j'\longmapsto j-2n,
\qquad
\overline{j}\longmapsto j-n,
\qquad
j\longmapsto j.
\end{equation}
This correspondence forms a bijection 
\[
    F: \fB_n\longrightarrow \left\{f\in\dschub_n~:~i-2n<f(i)\le i \text{ for all }i\right\} \qquad \text{given by}\qquad  \pi\mapsto f_\pi,
\]
satisfying
\begin{itemize}
    \item[(i)] if $\nega{\pi}=k,~ \exc{\pi}=r$, then $\operatorname{disp}(f_{\pi})=-(2r+k)$;
    \item[(ii)] $\DEX{\pi}=\DES{\left.f_\pi\right|_{[n]}}\coloneqq\left\{i\in [n-1]~:~f_\pi(i)>f_{\pi}(i+1)\right\}$.
\end{itemize}

\begin{remark}
In the literature, the statistics $2\cdot\exc{\pi}+\nega{\pi}$ is known as the \emph{flag excedance number} $\operatorname{fexc}(\pi)$ of $\pi$.
It is one of the flag Eulerian statistics introduced in order to extend the Eulerian statistics in $\fS_n$ to $\fB_n$, or more generally to the $r$-colored permutation group $\ZZ_r\wr\fS_n$. 
\end{remark}



Recall from the introduction, given the bijection $F$, for $\pi\in\coloredperm{n}{k}$ with $\exc{\pi}=r$, the alignment number of $\pi$ is
\[
    \al{\pi}=\ell(f_{\pi})-r(n-k-r).
\]
For any affine permutation $f$ in the image of $F$, we define the Level sets of $f$ as follows:
\begin{align*}
    &\operatorname{Lev}_{-1}(f)\coloneqq\{i\in [n]~:~-n<f(i)\le 0\},\\
    &\operatorname{Lev}_{-2}(f)\coloneqq \{i\in [n]~:~i-2n<f(i)\le -n\}.
\end{align*}

Then our refinement of the bicolored Eulerian quasisymmetric function $\coloredfunc{n}{k}{r}(\xxx;q)$ can be interpreted solely in terms of affine permutations.

\begin{proposition}
For $n\geq1$, $0\leq k\leq n$, and $0\leq r\leq n-k$,
\[
    \coloredfunc{n}{k}{r}(\xxx;q)=\sum_{f} q^{\ell(f)-r(n-k-r)}\FF_{\DES{\left.f\right|_{[n]}},n}(\xxx),
\]
where the sum runs over $f\in\dschub_n$ satisfying $i-2n<f(i)\le i$ for all $i$, $|\operatorname{Lev}_{-1}(f)|=k$, and $|\operatorname{Lev}_{-2}(f)|=r$.
\end{proposition}

Next, we will give a direct combinatorial interpretation of the alignment number $\al{\pi}$ for $\pi\in\fB_n$ that justifies the name of this statistic.

We fix integers $n\geq 1$, $0\leq k\leq n$, and $0\leq r\leq n-k$, and
put
\[
u:=n-k-r
\]
throughout this section. For a proposition $\mathcal{Q}$, let $\one_{\mathcal{Q}}$ be $1$ if
$\mathcal{Q}$ is true and $0$ otherwise.

\begin{definition}\label{def:admissible-inverse-data}
Let $\operatorname{Adm}_{n,k,r}$ be the set of pairs $(\eta,C)$ with
the following properties:
\begin{enumerate}
    \item $\eta\in\fS_n$;
    \item $C=C_1C_2\cdots C_n$ is a word in the alphabet
    $\{\mathsf{P},\mathsf{B},\mathsf{N}\}$ containing exactly $r$
    letters $\mathsf{P}$, exactly $k$ letters $\mathsf{B}$, and exactly
    $u$ letters $\mathsf{N}$;
    \item for every $i\in[n]$,
    \[
    C_i=\mathsf{P}\Longrightarrow\eta_i<i,
    \qquad
    C_i=\mathsf{N}\Longrightarrow\eta_i\geq i.
    \]
    No inequality is imposed on $\eta_i$ when $C_i=\mathsf{B}$.
\end{enumerate}
For such a word $C$, define the three sets of values
\begin{align*}
\operatorname{Prim}(C)&:=\{i\in[n]:C_i=\mathsf{P}\},\\
\operatorname{Bar}(C)&:=\{i\in[n]:C_i=\mathsf{B}\},\\
\operatorname{NonExc}(C)&:=\{i\in[n]:C_i=\mathsf{N}\}.
\end{align*}
Thus these three sets form a disjoint union of $[n]$ and have
cardinalities $r$, $k$, and $u$, respectively.
\end{definition}

The letters $\mathsf{P}$, $\mathsf{B}$, and $\mathsf{N}$ stand for
\emph{primed}, \emph{barred}, and \emph{nonexcedance}, respectively. We say \emph{``$i$ is of type $\mathsf{T}$''} ($\mathsf{T}\in\mathsf{P},\mathsf{B},\mathsf{N}$) if $C_i=\mathsf{T}$.
Notice that the index $i$ in $C_i$ will correspond to a value, not a position in the
one-line notation of a bicolored permutation (see the bijection below).

We say that ``$(\eta,C)$ is \emph{admissible}'' if it satisfies (1) and (3) in Definition~\ref{def:admissible-inverse-data}.

\begin{proposition}\label{prop:bicolored-admissible-bijection}
There is a bijection
\[
\left\{\pi\in\coloredperm{n}{k}:\exc{\pi}=r\right\}
\longrightarrow \operatorname{Adm}_{n,k,r}.
\]
\end{proposition}

\begin{proof}
Let $\pi\in\coloredperm{n}{k}$ satisfy $\exc{\pi}=r$.  Let
$\sigma=\sigma_1\cdots\sigma_n\in\fS_n$ be the ordinary
permutation obtained from $\pi$ by erasing all bars, and define the set
of barred values by
\[
\operatorname{BarVal}(\pi)
:=\{\sigma_j:j\in[n]\text{ and }\pi_j=\overline{\sigma_j}\}.
\]
In particular, $\operatorname{BarVal}(\pi)$ is a set of values; it is
not the set of positions at which bars occur.  Set
\[
\eta:=\sigma^{-1}.
\]
Thus, for every $i\in[n]$, the integer $\eta_i$ is the unique position
occupied by the value $i$ in $\sigma$, so that
\[
\sigma_{\eta_i}=i.
\]
Define $C=C_1\cdots C_n$ by
\begin{align}
C_i=
\begin{cases}
\mathsf{B},&i\in\operatorname{BarVal}(\pi),\\
\mathsf{P},&i\notin\operatorname{BarVal}(\pi)\text{ and }\eta_i<i,\\
\mathsf{N},&i\notin\operatorname{BarVal}(\pi)\text{ and }\eta_i\geq i.
\end{cases}
\label{eq:color-word-from-pi}
\end{align}
There are exactly $k$ letters $\mathsf{B}$ in $C$.  If an unbarred
value $i$ occurs in position $j$, then $j=\eta_i$, and this occurrence
is an excedance precisely when
\[
i>j
\quad\Longleftrightarrow\quad
i>\eta_i
\quad\Longleftrightarrow\quad
C_i=\mathsf{P}.
\]
It follows that $C$ has exactly $r$ letters $\mathsf{P}$.  Its remaining
$n-k-r=u$ letters are $\mathsf{N}$.  The inequalities required in
Definition~\ref{def:admissible-inverse-data} follow directly from
\eqref{eq:color-word-from-pi}.  Hence $(\eta,C)$ belongs to
$\operatorname{Adm}_{n,k,r}$.

Conversely, let $(\eta,C)\in\operatorname{Adm}_{n,k,r}$ and put
$\sigma:=\eta^{-1}$.  Starting from the one-line notation
$\sigma_1\cdots\sigma_n$, bar the entry $\sigma_j$ if and only if
\[
C_{\sigma_j}=\mathsf{B}.
\]
This gives a bicolored permutation $\pi$ with exactly $k$ barred
entries.  If $C_i=\mathsf{P}$, then the value $i$ occurs in position
$\eta_i<i$ and is unbarred, so it gives an excedance of $\pi$.  If
$C_i=\mathsf{N}$, then $i$ is unbarred and occurs in position
$\eta_i\geq i$, so it is not an excedance.  Values of type $\mathsf{B}$
are barred and hence are not counted by $\EXC{\pi}$.  Therefore
$\exc{\pi}=r$.

Starting with $\pi$ and yielding $(\eta,C)$, the inverse construction $(\eta,C)\mapsto\widetilde{\pi}$ recovers
$\sigma=\eta^{-1}$ and bars precisely the values in
$\operatorname{BarVal}(\pi)$ and thus yields $\widetilde{\pi}=\pi$.  Starting with $(\eta,C)$ and yielding $\pi$, the inverse construction $\pi\mapsto(\widetilde{\eta},\widetilde{C})$ recovers each letter of $C$ because of formula
\eqref{eq:color-word-from-pi}, so $\widetilde{C}=C$. Additionally, it is clear that $\widetilde{\eta}=\eta$. Thus the two
constructions are mutually inverse.
\end{proof}

\begin{definition}\label{def:cr}
Let $(\eta,C)\in\operatorname{Adm}_{n,k,r}$.  Define its \emph{crossing} by
\begin{align*}
\operatorname{cr}_{\mathsf{B}}(\eta,C)
:={}&
\#\{(p_1,p_2)\in\operatorname{Prim}(C)^2:
       \eta_{p_1}<\eta_{p_2}<p_1<p_2\}\\
&+\#\{(j_1,j_2)\in\operatorname{NonExc}(C)^2:
       j_1<j_2\leq\eta_{j_1}<\eta_{j_2}\}\\
&+\#\{(b_1,b_2)\in\operatorname{Bar}(C)^2:
       b_1<b_2\text{ and }\eta_{b_1}<\eta_{b_2}\}\\
&+\#\{(b,j)\in\operatorname{Bar}(C)\times\operatorname{NonExc}(C):
       j\leq\eta_b<\eta_j\}\\
&+\#\{(b,p)\in\operatorname{Bar}(C)\times\operatorname{Prim}(C):
       \eta_p<b<p\}.
\end{align*}
Define its \emph{alignment} by
\begin{align*}
\operatorname{al}_{\mathsf{B}}(\eta,C)
:={}&
\#\{(p_1,p_2)\in\operatorname{Prim}(C)^2:
       p_1<p_2\text{ and }\eta_{p_1}>\eta_{p_2}\}\\
&+\#\{(b_1,b_2)\in\operatorname{Bar}(C)^2:
       b_1<b_2\text{ and }\eta_{b_1}>\eta_{b_2}\}\\
&+\#\{(j_1,j_2)\in\operatorname{NonExc}(C)^2:
       j_1<j_2\text{ and }\eta_{j_1}>\eta_{j_2}\}\\
&+\sum_{\substack{p\in\operatorname{Prim}(C)\\
                   b\in\operatorname{Bar}(C)}}
  \bigl(\one_{p<b}+\one_{\eta_p>\eta_b}\bigr)\\
&+\sum_{\substack{p\in\operatorname{Prim}(C)\\
                   j\in\operatorname{NonExc}(C)}}
  \bigl(\one_{p<j}+\one_{\eta_p>\eta_j}\bigr)\\
&+\sum_{\substack{b\in\operatorname{Bar}(C)\\
                   j\in\operatorname{NonExc}(C)}}
  \bigl(\one_{b<j}+\one_{\eta_b>\eta_j}\bigr).
\end{align*}
The two indicator functions attached to a pair of different types are
regarded as two distinct alignment incidences.  Consequently, such a
pair may contribute $0$, $1$, or $2$ to
$\operatorname{al}_{\mathsf{B}}(\eta,C)$.

If $\pi$ corresponds to $(\eta,C)$ under
Proposition~\ref{prop:bicolored-admissible-bijection}, define
\[
\cro{\pi}:=\operatorname{cr}_{\mathsf{B}}(\eta,C).
\]
\end{definition}

In particular, both statistics on admissible data in
Definition~\ref{def:cr} are
manifestly nonnegative integer-valued statistics.
\begin{remark}[A disk interpretation of the crossing statistic]
\label{rem:disk-crossing-diagram}
Let $\pi$ have underlying permutation $\sigma$, and put
$\eta=\sigma^{-1}$, so $\eta_v$ is the position of the value $v$. Let $C$ be the word given by \eqref{eq:color-word-from-pi}. 
Place two successive clockwise copies of
\[
\mathrm{tgt}_1,\mathrm{src}_1,\ldots,
\mathrm{tgt}_n,\mathrm{src}_n
\]
on a circle, called segments $\mathrm I$ and $\mathrm {II}$.
Sources $\mathrm{src}_i$ represent positions and targets $\mathrm{tgt}_i$ represent values; all
$4n$ points are distinct, with $\mathrm{tgt}_i$ preceding
$\mathrm{src}_i$ in each copy. Draw straight chords as follows:
\[
\begin{array}{c|c}
\text{type of the value}&\text{endpoints}\\ \hline
v\in\operatorname{Prim}(C)&
\mathrm{src}^{\mathrm I}_{\eta_v},\ \mathrm{tgt}^{\mathrm I}_v\\
b\in\operatorname{Bar}(C)&
\mathrm{src}^{\mathrm {II}}_{\eta_b},\ \mathrm{tgt}^{\mathrm I}_b\\
j\in\operatorname{NonExc}(C)&
\mathrm{src}^{\mathrm {II}}_{\eta_j},\ \mathrm{tgt}^{\mathrm {II}}_j
\end{array}
\]
Then $\cro{\pi}$ counts unordered pairs of chords
whose relative interiors intersect. Thus three chords intersecting at a point count as three intersecting pairs, not one intersection point.

We provide an example below. For $\pi=(\overline4,3,1,2)$, we have $\eta=(3,4,2,1)$.
Figure~\ref{fig:disk-crossing-compact} gives four successive stages:
(a) separate the target and source points;
(b) join position $2$ to value $3$ in segment $\mathrm I$;
(c) join positions $3,4$ to values $1,2$, respectively, in
segment $\mathrm {II}$;
(d) join position $1$ in segment $\mathrm {II}$ to value $4$
in segment $\mathrm I$.
There are two intersecting pairs, so $\cro{\pi}=2$.
Open circles are sources, squares are targets, and unused points
are pale gray.

\begin{figure}[H]
\centering
\begin{minipage}[t]{.48\linewidth}
\centering
\resizebox{.76\linewidth}{!}{\CrDiskPanel{1}}\\[1pt]
\footnotesize (a) Separate values and positions.
\end{minipage}\hfill
\begin{minipage}[t]{.48\linewidth}
\centering
\resizebox{.76\linewidth}{!}{\CrDiskPanel{3}}\\[1pt]
\footnotesize (b) Add the primed chord.
\end{minipage}

\smallskip
\begin{minipage}[t]{.48\linewidth}
\centering
\resizebox{.76\linewidth}{!}{\CrDiskPanel{5}}\\[1pt]
\footnotesize (c) Add the nonexcedance chords.
\end{minipage}\hfill
\begin{minipage}[t]{.48\linewidth}
\centering
\resizebox{.76\linewidth}{!}{\CrDiskPanel{7}}\\[1pt]
\footnotesize (d) Add the barred chord.
\end{minipage}
\caption{Constructing the disk diagram for $(\overline4,3,1,2)$.}
\label{fig:disk-crossing-compact}
\end{figure}
\end{remark}

\begin{lemma}\label{lem:pair-contributions}
For $(\eta,C)\in\operatorname{Adm}_{n,k,r}$, the combined contribution
of a pair to
$\operatorname{cr}_{\mathsf{B}}(\eta,C)+
 \operatorname{al}_{\mathsf{B}}(\eta,C)$
is given by the following table (i.e. $\operatorname{cr}_{\mathsf{B}}(\eta,C)+
 \operatorname{al}_{\mathsf{B}}(\eta,C)$ equals the sum of numbers in the second column over pairs in the first column of the table below):
\begin{equation}\label{eq:pair-contribution-table}
\begin{array}{c|c}
\text{type of pair}&\text{combined contribution}\\ \hline
p_1<p_2,\quad p_1,p_2\in\operatorname{Prim}(C)
    &\one_{\eta_{p_2}<p_1}\\
j_1<j_2,\quad j_1,j_2\in\operatorname{NonExc}(C)
    &\one_{j_2\leq\eta_{j_1}}\\
b_1<b_2,\quad b_1,b_2\in\operatorname{Bar}(C)
    &1\\
p\in\operatorname{Prim}(C),\quad b\in\operatorname{Bar}(C)
    &\one_{\eta_p<b}+\one_{\eta_b<\eta_p}\\
b\in\operatorname{Bar}(C),\quad j\in\operatorname{NonExc}(C)
    &\one_{b<j}+\one_{\eta_b\geq j}\\
p\in\operatorname{Prim}(C),\quad j\in\operatorname{NonExc}(C)
    &\one_{p<j}+\one_{\eta_p>\eta_j}.
\end{array}
\end{equation}
\end{lemma}

\begin{proof}
First of all, every pair in Definition~\ref{def:cr} occurs once in the table~\eqref{eq:pair-contribution-table}, so the pairs not in this table have no contribution to $\operatorname{cr}_{\mathsf{B}}(\eta,C)+
 \operatorname{al}_{\mathsf{B}}(\eta,C)$ and it suffices to deal with all the pairs in this table one by one.
 
Fix $p_1<p_2$ in $\operatorname{Prim}(C)$.  The crossing and alignment
contributions of this pair are, respectively,
\[
\one_{\eta_{p_1}<\eta_{p_2}<p_1}
\qquad\text{and}\qquad
\one_{\eta_{p_1}>\eta_{p_2}}.
\]
Since $\eta_{p_1}<p_1$, the condition
$\eta_{p_1}>\eta_{p_2}$ implies $\eta_{p_2}<p_1$.  Conversely, if
$\eta_{p_2}<p_1$, then exactly one of
$\eta_{p_1}<\eta_{p_2}$ and $\eta_{p_1}>\eta_{p_2}$ holds.  In the
first case the pair is counted by the crossing term, and in the second
case it is counted by the alignment term.  This proves the first row.

Fix $j_1<j_2$ in $\operatorname{NonExc}(C)$.  The crossing and alignment
contributions are
\[
\one_{j_2\leq\eta_{j_1}<\eta_{j_2}}
\qquad\text{and}\qquad
\one_{\eta_{j_1}>\eta_{j_2}}.
\]
Because $\eta_{j_2}\geq j_2$, the second condition implies
$\eta_{j_1}>\eta_{j_2}\geq j_2$.  Conversely, if
$j_2\leq\eta_{j_1}$, then exactly one of
$\eta_{j_1}<\eta_{j_2}$ and $\eta_{j_1}>\eta_{j_2}$ holds; these two
cases give the crossing and alignment contributions, respectively.
This proves the second row.

If $b_1<b_2$ belong to $\operatorname{Bar}(C)$, then their crossing and
alignment contributions are
\[
\one_{\eta_{b_1}<\eta_{b_2}}
\qquad\text{and}\qquad
\one_{\eta_{b_1}>\eta_{b_2}}.
\]
The entries $\eta_{b_1}$ and $\eta_{b_2}$ are distinct, because $\eta$
is a permutation.  Exactly one of the two inequalities holds, proving
the third row.

For $p\in\operatorname{Prim}(C)$ and
$b\in\operatorname{Bar}(C)$, the crossing contribution is
$\one_{\eta_p<b<p}$ and the alignment contribution consists of
$\one_{p<b}$ and $\one_{\eta_p>\eta_b}$.  Since $\eta_p<p$ and
$b\neq p$, the two conditions $p<b$ and $\eta_p<b<p$ are disjoint and
their union is the condition $\eta_p<b$.  Hence
\[
\one_{p<b}+\one_{\eta_p<b<p}=\one_{\eta_p<b}.
\]
Moreover, $\eta_p\neq\eta_b$, so
$\one_{\eta_p>\eta_b}=\one_{\eta_b<\eta_p}$.  This proves the fourth
row.

For $b\in\operatorname{Bar}(C)$ and
$j\in\operatorname{NonExc}(C)$, the crossing contribution is
$\one_{j\leq\eta_b<\eta_j}$, while the alignment contribution consists of
$\one_{b<j}$ and $\one_{\eta_b>\eta_j}$.  Since $\eta_j\geq j$ and
$\eta_b\neq\eta_j$, the conditions
\[
j\leq\eta_b<\eta_j
\qquad\text{and}\qquad
\eta_b>\eta_j
\]
are disjoint and their union is $\eta_b\geq j$.  Therefore
\[
\one_{j\leq\eta_b<\eta_j}+\one_{\eta_b>\eta_j}
=\one_{\eta_b\geq j},
\]
which proves the fifth row after adding $\one_{b<j}$.

Finally, no term in the definition of
$\operatorname{cr}_{\mathsf{B}}$ is indexed by a
$\mathsf{P}\mathsf{N}$ pair.  Its combined contribution is therefore
exactly its alignment contribution
$\one_{p<j}+\one_{\eta_p>\eta_j}$.  This proves the sixth row.
\end{proof}

Define
\begin{equation}\label{eq:delta-definition}
\delta_{n,k,r}
:=\binom{n}{2}-\binom{r+1}{2}-\binom{u}{2},
\qquad u=n-k-r.
\end{equation}

\begin{theorem}\label{thm:cr-al-complement}
For every $(\eta,C)\in\operatorname{Adm}_{n,k,r}$,
\[
\operatorname{cr}_{\mathsf{B}}(\eta,C)
+\operatorname{al}_{\mathsf{B}}(\eta,C)
=\delta_{n,k,r}.
\]
\end{theorem}

\begin{proof}
All pairs of values belong to exactly one of the six rows in
\eqref{eq:pair-contribution-table}.  We sum the contributions in those
six rows.

First, every pair of values in $\operatorname{Bar}(C)$ contributes
$1$.  Their total contribution is therefore
\begin{equation}\label{eq:bar-bar-total}
\binom{k}{2}.
\end{equation}

We next collect the contributions of all
$\mathsf{P}\mathsf{P}$, $\mathsf{P}\mathsf{B}$, and
$\mathsf{P}\mathsf{N}$ pairs.  For $p\in\operatorname{Prim}(C)$, set
\begin{align*}
\operatorname{contrib}_{\mathsf{P}}(p)
:={}&
\#\{p'\in\operatorname{Prim}(C):\eta_p<p'<p\}\\
&+\#\{b\in\operatorname{Bar}(C):\eta_p<b\}
 +\#\{b\in\operatorname{Bar}(C):\eta_b<\eta_p\}\\
&+\#\{j\in\operatorname{NonExc}(C):p<j\}
 +\#\{j\in\operatorname{NonExc}(C):\eta_j<\eta_p\}.
\end{align*}
In the first term, $p$ is the larger of the two
$\operatorname{Prim}(C)$-values.  It follows from the first, fourth, and
sixth rows of \eqref{eq:pair-contribution-table} that the total
contribution of these three kinds of pairs is
\begin{equation}\label{eq:prim-pairs-as-sum}
\sum_{p\in\operatorname{Prim}(C)}
\operatorname{contrib}_{\mathsf{P}}(p).
\end{equation}

For $p\in\operatorname{Prim}(C)$, define
\[
\operatorname{gap}_{\mathsf{N}}(p)
:=\#\{j\in\operatorname{NonExc}(C):\eta_p<j\leq p\},
\]
and put
\begin{equation}\label{eq:gap-total-definition}
\operatorname{Gap}_{\mathsf{P},\mathsf{N}}(\eta,C)
:=\sum_{p\in\operatorname{Prim}(C)}
  \operatorname{gap}_{\mathsf{N}}(p).
\end{equation}
Since $\eta_p<p$, the two disjoint intervals
$\eta_p<j\leq p$ and $p<j$ together give the interval $\eta_p<j$.
Consequently,
\begin{align}\label{eq:prim-contribution-expanded}
&\operatorname{contrib}_{\mathsf{P}}(p)
 +\operatorname{gap}_{\mathsf{N}}(p)\\
 ={}&
\#\{p'\in\operatorname{Prim}(C):\eta_p<p'<p\}\notag\\
&+(\#\{b\in\operatorname{Bar}(C):
                 \eta_p<b\} + \#\{j\in\operatorname{NonExc}(C):p<j\}+\#\{j\in\operatorname{NonExc}(C):\eta_p<j\leq p\})\notag\\
&+(\#\{b\in\operatorname{Bar}(C):\eta_b<\eta_p\} + \#\{j\in\operatorname{NonExc}(C):\eta_j<\eta_p\}\notag \\
={}&
\#\{p'\in\operatorname{Prim}(C):\eta_p<p'<p\}\notag\\
&+\#\{x\in\operatorname{Bar}(C)\cup\operatorname{NonExc}(C):
                 x>\eta_p\}\notag\\
&+\#\{x\in\operatorname{Bar}(C)\cup\operatorname{NonExc}(C):
                 \eta_x<\eta_p\}.\notag
\end{align}
There are $n-\eta_p$ values in $[n]$ that are strictly larger than
$\eta_p$.  Among them, exactly
\[
r-\#\{p'\in\operatorname{Prim}(C):p'\leq\eta_p\}
\]
belong to $\operatorname{Prim}(C)$.  Hence
\begin{align}
&\#\{x\in\operatorname{Bar}(C)\cup\operatorname{NonExc}(C):
                 x>\eta_p\}\notag\\
={}&n-\eta_p-r
 +\#\{p'\in\operatorname{Prim}(C):p'\leq\eta_p\}\notag\\
={}&k+u-\eta_p
 +\#\{p'\in\operatorname{Prim}(C):p'\leq\eta_p\}.
\label{eq:values-above-etap}
\end{align}
Because $\eta$ is a permutation, exactly $\eta_p-1$ values
$x\in[n]$ satisfy $\eta_x<\eta_p$.  Removing the values in
$\operatorname{Prim}(C)$ gives
\begin{align}
&\#\{x\in\operatorname{Bar}(C)\cup\operatorname{NonExc}(C):
                 \eta_x<\eta_p\}\notag\\
={}&\eta_p-1
 -\#\{p'\in\operatorname{Prim}(C):\eta_{p'}<\eta_p\}.
\label{eq:positions-below-etap}
\end{align}
Substituting \eqref{eq:values-above-etap} and
\eqref{eq:positions-below-etap} into
\eqref{eq:prim-contribution-expanded}, and using $\eta_p<p$, gives
\begin{align}
&\operatorname{contrib}_{\mathsf{P}}(p)
 +\operatorname{gap}_{\mathsf{N}}(p)\notag\\
={}&k+u-1
 +\#\{p'\in\operatorname{Prim}(C):p'<p\}
 -\#\{p'\in\operatorname{Prim}(C):\eta_{p'}<\eta_p\}.
\label{eq:prim-contribution-ranks}
\end{align}
Indeed, the first cardinality after the last equal sign in
\eqref{eq:prim-contribution-expanded} and the last cardinality in
\eqref{eq:values-above-etap} add to
\[
\#\{p'\in\operatorname{Prim}(C):p'<p\},
\]
because the disjoint intervals $p'\leq\eta_p$ and
$\eta_p<p'<p$ partition the $\operatorname{Prim}(C)$-values below $p$ (recall that $p\in\operatorname{Prim}(C)\Rightarrow\eta_p<p$ so that the inequality $p'\leq\eta_p$ can imply $p^\prime<p$).

As $p$ ranges over $\operatorname{Prim}(C)$, each unordered pair of
distinct $\operatorname{Prim}(C)$-values is counted exactly once in
\[
\sum_{p\in\operatorname{Prim}(C)}
\#\{p'\in\operatorname{Prim}(C):p'<p\}.
\]
Therefore this sum equals $\binom{r}{2}$.  The numbers
$\{\eta_p:p\in\operatorname{Prim}(C)\}$ are also distinct, so the same
argument, now using their usual order, gives
\[
\sum_{p\in\operatorname{Prim}(C)}
\#\{p'\in\operatorname{Prim}(C):\eta_{p'}<\eta_p\}
=\binom{r}{2}.
\]
Consequently, summing \eqref{eq:prim-contribution-ranks} over all
$p\in\operatorname{Prim}(C)$ and applying
\eqref{eq:gap-total-definition} yields
\[\sum_{p\in\operatorname{Prim}(C)}\operatorname{contrib}_{\mathsf{P}}(p)+\operatorname{Gap}_{\mathsf{P},\mathsf{N}}(\eta,C) = r(k+u-1) + \binom{r}{2} - \binom{r}{2} = r(k+u-1)\]
and thus
\begin{equation}\label{eq:prim-total}
\sum_{p\in\operatorname{Prim}(C)}
\operatorname{contrib}_{\mathsf{P}}(p)
=r(k+u-1)-\operatorname{Gap}_{\mathsf{P},\mathsf{N}}(\eta,C).
\end{equation}

It remains to count the $\mathsf{N}\mathsf{N}$ and
$\mathsf{B}\mathsf{N}$ contributions.  Let their total be
$\operatorname{contrib}_{\mathsf{N}}(\eta,C)$.  The second and fifth
rows of \eqref{eq:pair-contribution-table} give
\begin{align}
\operatorname{contrib}_{\mathsf{N}}(\eta,C)
={}&
\sum_{\substack{j_1,j_2\in\operatorname{NonExc}(C)\\j_1<j_2}}
\one_{j_2\leq\eta_{j_1}}\notag\\
&+\sum_{\substack{b\in\operatorname{Bar}(C)\\
                   j\in\operatorname{NonExc}(C)}}
\bigl(\one_{b<j}+\one_{\eta_b\geq j}\bigr).
\label{eq:nonexc-total-start}
\end{align}
There are $ku$ pairs in
$\operatorname{Bar}(C)\times\operatorname{NonExc}(C)$, and for each such
pair
\[
\one_{\eta_b\geq j}-1=-\one_{\eta_b<j}.
\]
Subtracting $ku$ from \eqref{eq:nonexc-total-start} therefore gives
\begin{align}
&\operatorname{contrib}_{\mathsf{N}}(\eta,C)-ku\notag\\
={}&
\sum_{\substack{j_1,j_2\in\operatorname{NonExc}(C)\\j_1<j_2}}
\one_{j_2\leq\eta_{j_1}}\notag\\
&+\sum_{\substack{b\in\operatorname{Bar}(C)\\
                   j\in\operatorname{NonExc}(C)}}
\bigl(\one_{b<j}-\one_{\eta_b<j}\bigr).
\label{eq:nonexc-total-minus-ku}
\end{align}
For every $b\in\operatorname{Bar}(C)$ and
$j\in\operatorname{NonExc}(C)$, we have
\begin{align}
\one_{b<j}-\one_{\eta_b<j}
={}&\one_{b<j,\ \eta_b\geq j}
    -\one_{\eta_b<j,\ b\geq j}\notag\\
={}&\one_{b<j\leq\eta_b}-\one_{\eta_b<j\leq b}.
\label{eq:bar-cut-difference}
\end{align}
Here the first equality follows by cancelling the case in which both
$b<j$ and $\eta_b<j$ hold, and the second uses $b\neq j$.

We now use the following cut-balance identity.  For every $\tau\in[n]$,
\begin{equation}\label{eq:cut-balance}
\#\{x\in[n]:x<\tau\leq\eta_x\}
=\#\{x\in[n]:\eta_x<\tau\leq x\}.
\end{equation}
To prove it, set
\[
A_\tau:=\{x\in[n]:x<\tau\},
\qquad
D_\tau:=\{x\in[n]:\eta_x<\tau\}.
\]
Then the left-hand side of \eqref{eq:cut-balance} is
$|A_\tau\setminus D_\tau|$, and the right-hand side is
$|D_\tau\setminus A_\tau|$.  We have $|A_\tau|=\tau-1$.  Since $\eta$
is a permutation, exactly $\tau-1$ of its values are smaller than
$\tau$, so $|D_\tau|=\tau-1$ as well.  The identity
\[
|A_\tau\setminus D_\tau|-|D_\tau\setminus A_\tau|
=|A_\tau|-|D_\tau|=0
\]
proves \eqref{eq:cut-balance}.

Apply \eqref{eq:cut-balance} with $\tau=j$ for a fixed
$j\in\operatorname{NonExc}(C)$.  A value
$p\in\operatorname{Prim}(C)$ cannot contribute to the left-hand side,
because $p<j$ and $\eta_p<p$ would imply $\eta_p<j$, contrary to
$j\leq\eta_p$.  A value $j'\in\operatorname{NonExc}(C)$ cannot contribute
to the right-hand side, because $\eta_{j'}\geq j'$ and $j\leq j'$ would
imply $j\leq j'\leq\eta_{j'}$, contrary to $\eta_{j'}<j$.  Splitting
the remaining values according to their types therefore gives
\begin{align}
&\#\{j'\in\operatorname{NonExc}(C):j'<j\leq\eta_{j'}\}
 +\#\{b\in\operatorname{Bar}(C):b<j\leq\eta_b\}\notag\\
={}&
\#\{p\in\operatorname{Prim}(C):\eta_p<j\leq p\}
 +\#\{b\in\operatorname{Bar}(C):\eta_b<j\leq b\}.
\label{eq:cut-balance-by-type}
\end{align}
Sum \eqref{eq:cut-balance-by-type} over
$j\in\operatorname{NonExc}(C)$.  The first term on its left is the first
sum on the right-hand side of \eqref{eq:nonexc-total-minus-ku}; the
difference between the second terms on its two sides is the second sum
in \eqref{eq:nonexc-total-minus-ku}, by
\eqref{eq:bar-cut-difference}; and the first term on its right is
$\operatorname{Gap}_{\mathsf{P},\mathsf{N}}(\eta,C)$.  It follows that
\begin{equation}\label{eq:nonexc-total}
\operatorname{contrib}_{\mathsf{N}}(\eta,C)
=ku+\operatorname{Gap}_{\mathsf{P},\mathsf{N}}(\eta,C).
\end{equation}

Finally, combine \eqref{eq:bar-bar-total}, \eqref{eq:prim-total}, and
\eqref{eq:nonexc-total}.  The two gap terms cancel, and hence
\begin{align*}
\operatorname{cr}_{\mathsf{B}}(\eta,C)
+\operatorname{al}_{\mathsf{B}}(\eta,C)
&=\binom{k}{2}+r(k+u-1)+ku\\
&=\binom{k}{2}+rk+ru+ku-r.
\end{align*}
Since $n=r+k+u$, expanding $\binom{n}{2}$ gives
\[
\binom{n}{2}
=\binom{r}{2}+\binom{k}{2}+\binom{u}{2}+rk+ru+ku.
\]
Using $\binom{r+1}{2}=\binom{r}{2}+r$, we obtain
\begin{align*}
\delta_{n,k,r}
&=\binom{n}{2}-\binom{r+1}{2}-\binom{u}{2}\\
&=\binom{k}{2}+rk+ru+ku-r.
\end{align*}
This proves the theorem.
\end{proof}

\begin{proposition}\label{prop:signed-formulas}
Let $(\eta,C)\in\operatorname{Adm}_{n,k,r}$.  Write the three sets of
values in increasing order as
\begin{align*}
\operatorname{Prim}(C)&=\{p_1<\cdots<p_r\},\\
\operatorname{Bar}(C)&=\{b_1<\cdots<b_k\},\\
\operatorname{NonExc}(C)&=\{j_1<\cdots<j_u\}.
\end{align*}
An empty set contributes an empty block in the following definitions.
Define two permutations of $[n]$ by
\begin{align*}
\operatorname{ValWord}(C)
&:=(p_1,\ldots,p_r,b_1,\ldots,b_k,j_1,\ldots,j_u),\\
\operatorname{PosWord}(\eta,C)
&:=(\eta_{p_1},\ldots,\eta_{p_r},
    \eta_{b_1},\ldots,\eta_{b_k},
    \eta_{j_1},\ldots,\eta_{j_u}).
\end{align*}
Then
\begin{align}
\operatorname{cr}_{\mathsf{B}}(\eta,C)
&=\binom{k}{2}-r
  +\inv\bigl(\operatorname{ValWord}(C)\bigr)
  -\inv\bigl(\operatorname{PosWord}(\eta,C)\bigr),
\label{eq:signed-cr}\\
\operatorname{al}_{\mathsf{B}}(\eta,C)
&=rk+ru+ku
  +\inv\bigl(\operatorname{PosWord}(\eta,C)\bigr)
  -\inv\bigl(\operatorname{ValWord}(C)\bigr).
\label{eq:signed-al}
\end{align}
\end{proposition}

\begin{proof}
We first prove \eqref{eq:signed-al}.  Consider a pair $x<y$ of values
of the same type (i.e. $x,y\in T$ for some $T\in\{\operatorname{Prim}(C),\operatorname{Bar}(C),\operatorname{NonExc}(C)\}$).  The value $x$ occurs before $y$ in
$\operatorname{ValWord}(C)$, and the corresponding entries in
$\operatorname{PosWord}(\eta,C)$ are $\eta_x$ and $\eta_y$.  Thus this
pair contributes $0$ to $\inv(\operatorname{ValWord}(C))$ and
$\one_{\eta_x>\eta_y}$ to
$\inv(\operatorname{PosWord}(\eta,C))$.  This is exactly the
same-type contribution of the pair to
$\operatorname{al}_{\mathsf{B}}(\eta,C)$.

Now let $x$ and $y$ have different types, with the type of $x$ preceding
the type of $y$ in the order
\[
\mathsf{P}<\mathsf{B}<\mathsf{N}.
\]
The value $x$ occurs before $y$ in both words.  Hence this pair
contributes
\[
\one_{\eta_x>\eta_y}-\one_{x>y}
\]
to
\[
\inv\bigl(\operatorname{PosWord}(\eta,C)\bigr)
-\inv\bigl(\operatorname{ValWord}(C)\bigr).
\]
On the other hand, the definition of
$\operatorname{al}_{\mathsf{B}}$ assigns to this pair the contribution
\[
\one_{x<y}+\one_{\eta_x>\eta_y}.
\]
Since $x\neq y$, we have
$\one_{x<y}=1-\one_{x>y}$, and consequently
\[
\one_{x<y}+\one_{\eta_x>\eta_y}
=1+\one_{\eta_x>\eta_y}-\one_{x>y}.
\]
There are $rk$ pairs of types $\mathsf{P}\mathsf{B}$, $ru$ pairs of
types $\mathsf{P}\mathsf{N}$, and $ku$ pairs of types
$\mathsf{B}\mathsf{N}$.  Summing the same-type and different-type
contributions over all pairs proves \eqref{eq:signed-al}.

By Theorem~\ref{thm:cr-al-complement} and
\eqref{eq:signed-al},
\begin{align*}
\operatorname{cr}_{\mathsf{B}}(\eta,C)
&=\delta_{n,k,r}-\operatorname{al}_{\mathsf{B}}(\eta,C)\\
&=\binom{k}{2}+rk+ru+ku-r-(rk+ru+ku)\\
&\qquad
  -\inv\bigl(\operatorname{PosWord}(\eta,C)\bigr)
  +\inv\bigl(\operatorname{ValWord}(C)\bigr)\\
&=\binom{k}{2}-r
  +\inv\bigl(\operatorname{ValWord}(C)\bigr)
  -\inv\bigl(\operatorname{PosWord}(\eta,C)\bigr).
\end{align*}
This is \eqref{eq:signed-cr}.
\end{proof}

\begin{proposition}\label{prop:affine-length-and-al}
Let $(\eta,C)\in\operatorname{Adm}_{n,k,r}$.  Define the level of a
value $i\in[n]$ by
\[
\operatorname{lev}_C(i):=
\begin{cases}
2,&C_i=\mathsf{P},\\
1,&C_i=\mathsf{B},\\
0,&C_i=\mathsf{N}.
\end{cases}
\]
Define an extended affine permutation $f_{\eta,C}:\mathbb{Z}\to
\mathbb{Z}$ by specifying its window as
\begin{equation}\label{eq:affine-window-from-data}
f_{\eta,C}(\eta_i):=i-n\operatorname{lev}_C(i)
\qquad(i\in[n])
\end{equation}
and extending periodically by
\[
f_{\eta,C}(j+n)=f_{\eta,C}(j)+n
\qquad(j\in\mathbb{Z}).
\]
Then
\begin{align}
\ell(f_{\eta,C})
&=rk+2ru+ku
 +\inv\bigl(\operatorname{PosWord}(\eta,C)\bigr)
 -\inv\bigl(\operatorname{ValWord}(C)\bigr),
\label{eq:affine-length-signed}\\
\ell(f_{\eta,C})
&=ru+\operatorname{al}_{\mathsf{B}}(\eta,C).
\label{eq:affine-length-al}
\end{align}
Consequently,
\begin{equation}\label{eq:affine-length-cr}
\ell(f_{\eta,C})
=ru+\delta_{n,k,r}-\operatorname{cr}_{\mathsf{B}}(\eta,C).
\end{equation}
If $\pi$ corresponds to $(\eta,C)$ under
Proposition~\ref{prop:bicolored-admissible-bijection}, then $f_{\eta,C}=f_\pi$ and
\begin{equation}\label{eq:al-positive-formula}
\al{\pi}=\operatorname{al}_{\mathsf{B}}(\eta,C).
\end{equation}
\end{proposition}

\begin{proof}
The values on the right-hand side of
\eqref{eq:affine-window-from-data} are pairwise incongruent modulo $n$,
because their residues are $1,2,\ldots,n$.  Hence the periodic extension
is a bijection of $\mathbb{Z}$ and is an affine permutation.
Its Coxeter length is given by the standard window formula
\begin{equation}\label{eq:affine-window-length-formula}
\ell(f_{\eta,C})
=\sum_{1\leq j_1<j_2\leq n}
\left|
\left\lfloor
\frac{f_{\eta,C}(j_2)-f_{\eta,C}(j_1)}{n}
\right\rfloor
\right|.
\end{equation}

Write
\[
\operatorname{ValWord}(C)=(v_1,\ldots,v_n).
\]
Then
\[
\operatorname{PosWord}(\eta,C)
=(\eta_{v_1},\ldots,\eta_{v_n}).
\]
Fix $1\leq a_1<a_2\leq n$, and put
\[
x:=v_{a_1},\qquad y:=v_{a_2},\qquad
X:=\eta_x,\qquad Y:=\eta_y,
\]
and
\[
\kappa:=\operatorname{lev}_C(x)-\operatorname{lev}_C(y).
\]
The construction of $\operatorname{ValWord}(C)$ implies
$\kappa\geq0$.  Moreover, if $\kappa=0$, then $x$ and $y$ have the
same type and hence $x<y$.

Suppose first that $X<Y$.  The contribution of the positions $X,Y$ to
\eqref{eq:affine-window-length-formula} is
\begin{align*}
\left|
\left\lfloor\frac{f_{\eta,C}(Y)-f_{\eta,C}(X)}{n}\right\rfloor
\right|
&=\left|
\left\lfloor\frac{y-x}{n}+\kappa\right\rfloor
\right|\\
&=\left|\kappa-\one_{x>y}\right|\\
&=\kappa-\one_{x>y}.
\end{align*}
The second equality uses $-(n-1)\leq y-x\leq n-1$ and $x\neq y$.
The expression in the last absolute value is nonnegative: if
$\kappa=0$, then $x<y$, while if $\kappa>0$, then
$\kappa-\one_{x>y}\geq0$.

Suppose next that $X>Y$.  Now the smaller position is $Y$, and its pair
contribution is
\begin{align*}
\left|
\left\lfloor\frac{f_{\eta,C}(X)-f_{\eta,C}(Y)}{n}\right\rfloor
\right|
&=\left|
\left\lfloor\frac{x-y}{n}-\kappa\right\rfloor
\right|\\
&=\left|-\kappa-\one_{x<y}\right|\\
&=\kappa+\one_{x<y}\\
&=\kappa+1-\one_{x>y}.
\end{align*}
Combining the two cases, the contribution of the pair indexed by
$a_1<a_2$ is
\begin{equation}\label{eq:affine-pair-contribution}
\kappa+\one_{X>Y}-\one_{x>y}.
\end{equation}

We sum \eqref{eq:affine-pair-contribution} over all
$a_1<a_2$.  The sum of $\one_{X>Y}$ is
$\inv(\operatorname{PosWord}(\eta,C))$, and the sum of
$\one_{x>y}$ is $\inv(\operatorname{ValWord}(C))$.  A
$\mathsf{P}\mathsf{B}$ pair contributes $1$ to $\kappa$, a
$\mathsf{P}\mathsf{N}$ pair contributes $2$, and a
$\mathsf{B}\mathsf{N}$ pair contributes $1$; same-type pairs contribute
$0$.  Since there are respectively $rk$, $ru$, and $ku$ such pairs,
the total contribution of the $\kappa$-terms is $rk+2ru+ku$.  This
proves \eqref{eq:affine-length-signed}.

Subtracting \eqref{eq:signed-al} from
\eqref{eq:affine-length-signed} gives
\[
\ell(f_{\eta,C})-\operatorname{al}_{\mathsf{B}}(\eta,C)=ru,
\]
which is \eqref{eq:affine-length-al}.  Finally,
Theorem~\ref{thm:cr-al-complement} gives
\[
\operatorname{al}_{\mathsf{B}}(\eta,C)
=\delta_{n,k,r}-\operatorname{cr}_{\mathsf{B}}(\eta,C),
\]
and substituting this into \eqref{eq:affine-length-al} proves
\eqref{eq:affine-length-cr}.

If $\pi$ corresponds to $(\eta,C)$, then the construction preceding
\eqref{eq:affine-length-signed} shows that $f_\pi=f_{\eta,C}$.  Since
$u=n-k-r$, Definition~\ref{def:al} and
\eqref{eq:affine-length-al} give
\[
\al{\pi}
=\ell(f_\pi)-ru
=\ell(f_{\eta,C})-ru
=\operatorname{al}_{\mathsf{B}}(\eta,C),
\]
which proves \eqref{eq:al-positive-formula}.
\end{proof}

\begin{corollary}\label{cor:cr-al-complement-for-pi}
If $\pi\in\coloredperm{n}{k}$ satisfies $\exc{\pi}=r$, then
\[
\cro{\pi}+\al{\pi}=\delta_{n,k,r}.
\]
\end{corollary}

\begin{proof}
Let $(\eta,C)\in\operatorname{Adm}_{n,k,r}$ correspond to $\pi$.
By Definition~\ref{def:cr} and \eqref{eq:al-positive-formula},
\[
\cro{\pi}=\operatorname{cr}_{\mathsf{B}}(\eta,C),
\qquad
\al{\pi}=\operatorname{al}_{\mathsf{B}}(\eta,C).
\]
The result now follows from Theorem~\ref{thm:cr-al-complement}.
\end{proof}

\begin{proposition}\label{prop:barred-unbarred-relation}
Let $n\geq1$, $1\leq k\leq n$, and $0\leq r\leq n-k$, and set
$u=n-k-r$.  There is a bijection from the set of permutations
$\pi\in\coloredperm{n}{k}$ with $\exc{\pi}=r$ for which the value $1$
is barred to the set of permutations
$\pi^\prime\in\coloredperm{n}{k-1}$ with $\exc{\pi^\prime}=r$ for which the value
$1$ is unbarred.  Under this bijection $\pi\mapsto\pi^\prime$,
\begin{align}
\DEX{\pi^\prime}&=\DEX{\pi},
\label{eq:bar-unbar-dex}\\
\al{\pi}&=\al{\pi^{\prime}}+u,
\label{eq:bar-unbar-al}\\
\ell(f_\pi)&=\ell(f_{\pi^\prime})+u-r.
\label{eq:bar-unbar-affine-length}
\end{align}
Consequently,
\begin{equation}\label{eq:barred-unbarred-function-relation}
\Phi^{\mathrm{bar}}_{n,k,r}(\xxx;q)
=
q^{n-k-r}\Phi^{\mathrm{unb}}_{n,k-1,r}(\xxx;q),
\end{equation}
and hence
\begin{equation}\label{eq:full-reduced-function-relation}
\coloredfunc{n}{k}{r}(\xxx;q)
=
\Phi^{\mathrm{unb}}_{n,k,r}(\xxx;q)
+q^{n-k-r}\Phi^{\mathrm{unb}}_{n,k-1,r}(\xxx;q).
\end{equation}
For $k=0$, the barred-one part $\Phi^{\mathrm{bar}}_{n,0,r}(\xxx;q) = 0$ and thus
$\coloredfunc{n}{0}{r}=\Phi^{\mathrm{unb}}_{n,0,r}$.
\end{proposition}

\begin{proof}
Under the bijection $\pi\mapsto(\eta,C)$ in Proposition~\ref{prop:bicolored-admissible-bijection}, the value
$1$ is barred in $\pi$ precisely when $C_1=\mathsf{B}$.  If it is unbarred,
then $C_1$ cannot be $\mathsf{P}$, because the defining inequality
$\eta_1<1$ is impossible.  Thus the value $1$ is unbarred precisely
when $C_1=\mathsf{N}$.  It is therefore enough first to construct a
bijection
\begin{equation}\label{eq:bar-unbar-admissible-bijection}
\begin{split}
\{(\eta,C)\in\operatorname{Adm}_{n,k,r}:C_1=\mathsf{B}\}
\longrightarrow{}&
\{(\eta^\prime,C^\prime)\in\operatorname{Adm}_{n,k-1,r}:\nobreak
 C^\prime_1=\mathsf{N}\}.
\end{split}
\end{equation}

Fix a pair $(\eta,C)$ in the set on the left and write
\[
\operatorname{Bar}(C)=\{b_1<b_2<\cdots<b_k\}.
\]
Since $C_1=\mathsf{B}$, we have $b_1=1$.  Define $C^\prime$ by
\[
C^\prime_1=\mathsf{N},
\qquad
C^\prime_i=C_i\quad(2\leq i\leq n),
\]
and define $\eta^\prime$ by
\begin{align}
\eta^\prime_1&=\eta_{b_k},
\label{eq:bar-unbar-eta-prime-first}\\
\eta^\prime_{b_i}&=\eta_{b_{i-1}}
\qquad(2\leq i\leq k),
\label{eq:bar-unbar-eta-prime-bars}\\
\eta^\prime_x&=\eta_x
\qquad\bigl(x\notin\operatorname{Bar}(C)\bigr).
\label{eq:bar-unbar-eta-prime-other}
\end{align}
When $k=1$, these equations simply say that $C^\prime_1=\mathsf{N}$ and
$\eta^\prime=\eta$, with all other entries unchanged.

The entries of $\eta^\prime$ are obtained by cyclically permuting the entries
$\eta_{b_1},\ldots,\eta_{b_k}$ of $\eta$, so $\eta^\prime$ belongs to
$\fS_n$.  Every value of type $\mathsf{P}$ has the same type
and the same $\eta$-entry after the operation, and the same is true of
every value that was of type $\mathsf{N}$.  The only new
$\mathsf{N}$-value is $1$, for which the required inequality
$\eta^\prime_1\geq1$ holds automatically.  Hence $(\eta^\prime,C^\prime)$ is
admissible (i.e. $(\eta^\prime,C^\prime)$ satisfies conditions (1) and (3) in Definition~\ref{def:admissible-inverse-data}).  Its numbers of letters of types $\mathsf{P}$,
$\mathsf{B}$, and $\mathsf{N}$ are respectively $r$, $k-1$, and
$u+1$, so it lies in the set on the right of
\eqref{eq:bar-unbar-admissible-bijection}.

For completeness, the inverse map is explicit.  Given a pair on the
right of \eqref{eq:bar-unbar-admissible-bijection}, if $k\geq2$, list
\[
\operatorname{Bar}(C^\prime)=\{b_2<\cdots<b_k\}
\]
and put $b_1=1$.  If $k=1$, put $b_1=1$ and regard the list
$b_2,\ldots,b_k$ as empty.  Change $C^\prime_1=\mathsf{N}$ to
$C_1=\mathsf{B}$, leave all other letters unchanged, and set
\begin{align*}
\eta_{b_i}&=\eta^\prime_{b_{i+1}}
&& (1\leq i\leq k-1),\\
\eta_{b_k}&=\eta^\prime_1,\\
\eta_x&=\eta^\prime_x
&&\bigl(x\notin\{b_1,\ldots,b_k\}\bigr).
\end{align*}
If $k=1$, the displayed list of barred values $\operatorname{Bar}(C^\prime)$ is empty and we have $b_1=1$ and $\eta_1=\eta^\prime_1$.  In either case the values newly
declared barred are subject to no admissibility inequality, while all
remaining inequalities are unchanged.  Direct substitution in
\eqref{eq:bar-unbar-eta-prime-first}--
\eqref{eq:bar-unbar-eta-prime-other} shows that the two constructions
are inverse.  This proves that \eqref{eq:bar-unbar-admissible-bijection}
is a bijection.

We next verify the statistics.  Write
\begin{align*}
\operatorname{Prim}(C)&=\{p_1<\cdots<p_r\},\\
\operatorname{NonExc}(C)&=\{j_1<\cdots<j_u\}.
\end{align*}
The value $1$ belongs to neither set.  By
\eqref{eq:bar-unbar-eta-prime-first}--
\eqref{eq:bar-unbar-eta-prime-other},
\begin{align}\label{eq:posword}
\operatorname{PosWord}(\eta^\prime,C^\prime)
={}&(\eta_{p_1},\ldots,\eta_{p_r},
      \eta_{b_1},\eta_{b_2},\ldots,\eta_{b_k},
      \eta_{j_1},\ldots,\eta_{j_u})\\
={}&\operatorname{PosWord}(\eta,C).\notag
\end{align}
Indeed, in the first line the entries
$\eta_{b_1},\ldots,\eta_{b_{k-1}}$ form the new barred block of $\operatorname{PosWord}(\eta^\prime,C^\prime)$, and
$\eta_{b_k}=\eta^\prime_1$ is the first entry of the new nonexcedance block of $\operatorname{PosWord}(\eta^\prime,C^\prime)$.

Let $\rho\in\symm_n$ be the standardization of the word $\pi^D$ in the ordered
alphabet of Definition~\ref{def:DEX}, and write
\[
\operatorname{ValWord}(C)=(v_1,\ldots,v_n).
\]
The sequence $v_1,\ldots,v_n$ records the underlying values of the
letters of $\pi^D$ in the increasing alphabet order of Definition~\ref{def:DEX}, and the position of
the letter with underlying value $v_a$ is $\eta_{v_a}$.  Consequently,
for every $a\in[n]$,
\[
\rho_{\eta_{v_a}}=a.
\]
Since
$\operatorname{PosWord}(\eta,C)
=(\eta_{v_1},\ldots,\eta_{v_n})$, this is equivalent to
\[
\rho=\operatorname{PosWord}(\eta,C)^{-1}.
\]
Standardization preserves the descent set.  Applying the same argument
to $\pi^\prime$ and using the equality~\eqref{eq:posword} of the two position words gives
\[
\DEX{\pi}
=\DES{\operatorname{PosWord}(\eta,C)^{-1}}
=\DES{\operatorname{PosWord}(\eta^\prime,C^\prime)^{-1}}
=\DEX{\pi^\prime},
\]
which proves \eqref{eq:bar-unbar-dex}.

The word $\operatorname{ValWord}(C^\prime)$ is obtained from
$\operatorname{ValWord}(C)$ by moving the entry $1$ to the right past
the $k-1$ entries $b_2,\ldots,b_k$, without changing any other relative
order.  Since each $b_i>1$ ($i=2,\dots,k$), this creates exactly $k-1$ inversions.
Thus
\begin{equation}\label{eq:bar-unbar-valword-inversions}
\inv\bigl(\operatorname{ValWord}(C^\prime)\bigr)
=
\inv\bigl(\operatorname{ValWord}(C)\bigr)+k-1.
\end{equation}
Using \eqref{eq:signed-cr}, the equality~\eqref{eq:posword} of the position words, and
\eqref{eq:bar-unbar-valword-inversions}, we obtain
\begin{align*}
\operatorname{cr}_{\mathsf{B}}(\eta^\prime,C^\prime)
={}&
\binom{k-1}{2}-r
+\inv\bigl(\operatorname{ValWord}(C^\prime)\bigr)
 -\inv\bigl(\operatorname{PosWord}(\eta^\prime,C^\prime)\bigr)\\
={}&
\binom{k-1}{2}-r
+\inv\bigl(\operatorname{ValWord}(C)\bigr)+k-1
 -\inv\bigl(\operatorname{PosWord}(\eta,C)\bigr)\\
={}&
\binom{k}{2}-r
+\inv\bigl(\operatorname{ValWord}(C)\bigr)
 -\inv\bigl(\operatorname{PosWord}(\eta,C)\bigr)\\
={}&\operatorname{cr}_{\mathsf{B}}(\eta,C).
\end{align*}
The number of nonexcedance values in the target is $u+1$, and hence
\begin{align*}
\delta_{n,k-1,r}
&=
\binom{n}{2}-\binom{r+1}{2}-\binom{u+1}{2}\\
&=\delta_{n,k,r}-u.
\end{align*}
Theorem~\ref{thm:cr-al-complement} now yields
\begin{align*}
\operatorname{al}_{\mathsf{B}}(\eta,C)
&=\delta_{n,k,r}-\operatorname{cr}_{\mathsf{B}}(\eta,C)\\
&=\delta_{n,k-1,r}+u-
  \operatorname{cr}_{\mathsf{B}}(\eta^\prime,C^\prime)\\
&=\operatorname{al}_{\mathsf{B}}(\eta^\prime,C^\prime)+u.
\end{align*}
Equation~\eqref{eq:al-positive-formula} proves
\eqref{eq:bar-unbar-al}.  Since the source has $u$ nonexcedance values
and the target has $u+1$, Definition~\ref{def:al} gives
\begin{align*}
\ell(f_\pi)-ru
&=\ell(f_{\pi^\prime})-r(u+1)+u,
\end{align*}
which is equivalent to \eqref{eq:bar-unbar-affine-length}.

Finally, \eqref{eq:bar-unbar-dex} and \eqref{eq:bar-unbar-al} show that
the summand indexed by $\pi$ in \eqref{eq:barred-one-part} is $q^u$
times the summand indexed by $\pi^\prime$ in
\eqref{eq:unbarred-one-part}.  Summing over the bijection gives
\eqref{eq:barred-unbarred-function-relation}.  Equation
\eqref{eq:full-reduced-function-relation} follows from
\eqref{eq:one-bar-status-decomposition}; the assertion for $k=0$ is
immediate from the definition of the barred-one part.
\end{proof}

\begin{corollary}\label{cor:minimum-affine-length}
For every $(\eta,C)\in\operatorname{Adm}_{n,k,r}$ and every
corresponding bicolored permutation $\pi\in\coloredperm{n}{k}$, one has
\[
\ell(f_{\eta,C})\geq ru
\qquad\text{and}\qquad
\al{\pi}\geq0.
\]
If $\operatorname{Adm}_{n,k,r}$ is nonempty, then
\[
\min_{(\eta,C)\in\operatorname{Adm}_{n,k,r}}
\ell(f_{\eta,C})=ru.
\]
Equivalently, the minimum value of
$\operatorname{al}_{\mathsf{B}}$ on
$\operatorname{Adm}_{n,k,r}$ is $0$.  Under the bijection in
Proposition~\ref{prop:bicolored-admissible-bijection}, this is also the
statement that
\[
\min_{\substack{\pi\in\coloredperm{n}{k}\\\exc{\pi}=r}}
\al{\pi}=0.
\]
\end{corollary}

\begin{proof}
Equation~\eqref{eq:affine-length-al} and the nonnegative definition of
$\operatorname{al}_{\mathsf{B}}$ give
$\ell(f_{\eta,C})\geq ru$ for every admissible pair.
Equation~\eqref{eq:al-positive-formula} then gives
$\al{\pi}\geq0$ for every corresponding bicolored
permutation.

Since $\operatorname{Adm}_{n,k,r}$ is nonempty, not every value can have
type $\mathsf{P}$: indeed, $\eta_1<1$ is impossible.  Thus $k+u>0$.
Consider the bicolored permutation whose underlying permutation is
\[
\sigma=
(k+u+1,k+u+2,\ldots,n,
 u+1,u+2,\ldots,u+k,
 1,2,\ldots,u),
\]
where the middle block is barred and the other two blocks are unbarred.
Empty blocks are omitted.  Its associated admissible data satisfy
\begin{align*}
\operatorname{Prim}(C)&=\{k+u+1,\ldots,n\},
&\eta_{k+u+j}&=j &&(1\leq j\leq r),\\
\operatorname{Bar}(C)&=\{u+1,\ldots,u+k\},
&\eta_{u+j}&=r+j &&(1\leq j\leq k),\\
\operatorname{NonExc}(C)&=\{1,\ldots,u\},
&\eta_j&=r+k+j &&(1\leq j\leq u).
\end{align*}
The positions are increasing within each type.  Moreover, for every
pair of different types, the value of the earlier type in the order
$\mathsf{P}<\mathsf{B}<\mathsf{N}$ is larger, while its position is
smaller.  Hence every indicator occurring in the definition of
$\operatorname{al}_{\mathsf{B}}$ is zero.  Thus
$\operatorname{al}_{\mathsf{B}}(\eta,C)=0$, and
\eqref{eq:affine-length-al} gives $\ell(f_{\eta,C})=ru$.
Equation~\eqref{eq:al-positive-formula} then also gives
$\al{\pi}=0$ for the corresponding bicolored
permutation.
\end{proof}

\section{Proofs of the Theta formulas}
\label{sec:proofs-theta-formulas}

We prove Theorems~\ref{conj:unbarred-one-theta} and
\ref{conj:bicolored-theta}. Fix $n\geq1$, $k,r\geq0$ with $u=n-k-r\geq0$. We further assume that $u\ge 1$ until finishing the proof of Theorem~\ref{conj:unbarred-one-theta}. Write
\[
\operatorname{Red}_{n,k,r}
:=\{\pi\in\coloredperm{n}{k}:\exc{\pi}=r
       \text{ and the value $1$ is unbarred}\}.
\]
For $\pi\in\operatorname{Red}_{n,k,r}$, retain the notation
$\sigma=\sigma_1\dots\sigma_n\in\symm_n$ obtained from $\pi$ by erasing all the bars, $\eta=\sigma^{-1}$, and $C=C_1\dots C_n$ given by
\[C_i\coloneqq\begin{cases}
    \mathsf{B}, &\text{if the letter $i$ is barred in $\pi$} \\
    \mathsf{P}, &\text{if the letter $i$ is unbarred in $\pi$ and $\eta_i<i$} \\
    \mathsf{N}, &\text{if the letter $i$ is unbarred in $\pi$ and $\eta_i\ge i$}
\end{cases}\]
which occur in the bijection $\pi\mapsto(\eta,C)$ in
Proposition~\ref{prop:bicolored-admissible-bijection}. In particular,
$C_1=\mathsf{N}$. All sets of values of the same type in this section are the sets
$\operatorname{Prim}(C)$, $\operatorname{Bar}(C)$, and
$\operatorname{NonExc}(C)$ of Definition~\ref{def:admissible-inverse-data}.
The statistics $\cro{\pi}$ and $\al{\pi}$
retain exactly their earlier normalizations.

The key step is to construct a bijection from
$\operatorname{Red}_{n,k,r}$ to segmented permutations with
$u-1$ ascents and $r$ descents, matching crossings with
sminversions and DEX sets with complementary Split sets; see the next subsection for the definitions of segmented permutations, sminversions, and Split sets.
We factor this bijection through Lentfer's monomial codes:
our main construction encodes bicolored permutations by these
codes, after which Lentfer's insertion bijection
\cite{Lentfer2025} produces the required segmented permutations. This gives Theorem~\ref{conj:unbarred-one-theta}.
Theorem~\ref{conj:bicolored-theta} then follows from the relation between
the barred-one and unbarred-one parts.

\subsection{Segmented permutations and the known Theta expansion}
\label{subsec:theta-segmented-preliminaries}

All definitions and results in this subsection arise from Iraci--Nadeau--Vanden Wyngaerd~
\cite{MR4766722}. A \emph{segmented permutation} of $[n]$ is a permutation
$w_1\cdots w_n$ together with a choice of separators between consecutive
entries, which separates $w$ into \emph{nonempty} blocks. Its nonempty blocks are called segments. Only adjacent pairs
within the same segment contribute to its numbers of \emph{ascents} and \emph{descents},
denoted by $\operatorname{asc}(w)\coloneqq\#\{i\in[n-1]\,:\,\text{$w_i<w_{i+1}$, the positions $i,i+1$ belong to the same segment}\}$ and $\operatorname{des}(w)\coloneqq\#\{i\in[n-1]\,:\,\text{$w_i>w_{i+1}$, the positions $i,i+1$ belong to the same segment}\}$.
For $d_1,d_2\geq0$ with $d_1+d_2<n$, write
\[
\operatorname{SW}(1^n;d_1,d_2)
:=\{w:\operatorname{asc}(w)=d_1,
          \operatorname{des}(w)=d_2\}.
\]
Every member of this set has $n-d_1-d_2$ segments.

An entry of $w$ is \emph{thick} if it starts a segment or is the smaller
entry of an adjacent descent within its segment; it is \emph{thin}
otherwise. These are properties of the entries, rather than their positions. Define $\operatorname{Split}(w)\subseteq[n-1]$
by declaring that $i\in\operatorname{Split}(w)$ precisely in one of
the following cases:
\begin{enumerate}[label=\textup{(\roman*)}]
\item the entry $i$ is thick and the entry $i+1$ is thin;
\item both entries are thin, and $i$ occurs to the left of $i+1$;
\item both entries are thick, and $i+1$ occurs to the left of $i$.
\end{enumerate}
Let $\operatorname{Start}(w)\subseteq[n]$ be the set of positions
starting segments. Define
\begin{equation}\label{eq:theta-sminv-definition}
\begin{split}
\operatorname{sminv}(w):=\#\{(a,b):{}&1\leq a<b\leq n,\ w_a>w_b,\\
&b\in\operatorname{Start}(w)\text{ or }
 (b\notin\operatorname{Start}(w)\text{ and }w_{b-1}>w_a)\}.
\end{split}
\end{equation}
The pairs counted here are called \emph{sminversions}. This is the specialization
of the Smirnov inversion statistic in Iraci--Nadeau--Vanden Wyngaerd~\cite{MR4766722} to words with distinct entries.

We use the following result of Iraci--Nadeau--Vanden Wyngaerd, combining
\cite[Theorem~3.1 and Proposition~5.3]{MR4766722}:
\begin{equation}\label{eq:theta-known-smirnov-expansion}
\sum_{w\in\operatorname{SW}(1^n;d_1,d_2)}
z^{\operatorname{sminv}(w)}
\FF_{\operatorname{Split}(w),n}(\xxx)
=
\left.
\Theta_{e_{d_1}}\Theta_{e_{d_2}}
\widetilde H_{(n-d_1-d_2)}
\right|_{t=0,\,q=z}.
\end{equation}
Here $z$ is an auxiliary grading variable.

\subsection{A prefix code for the unbarred-one part}
\label{subsec:theta-prefix-code}

We define a set $\operatorname{Code}_n$ equal to the set of exponent sequences of monomials introduced by \cite[Definition 3.1]{Lentfer2025}. Let $\operatorname{Code}_n$ consist of the sequences
\[
\mathbf c=((\alpha_i,\beta_i,\gamma_i))_{i=1}^n,
\qquad \alpha_i\in\mathbb Z_{\geq0},\quad
\beta_i,\gamma_i\in\{0,1\},
\]
such that
\begin{equation}\label{eq:theta-code-bounds}
\begin{gathered}
(\alpha_1,\beta_1,\gamma_1)=(0,0,0),\\
\operatorname{cap}_i=\operatorname{cap}_i(\mathbf c)
:=i-1-\sum_{a=1}^i(\beta_a+\gamma_a),
\qquad 0\leq\alpha_i\leq\operatorname{cap}_i
\quad(1\leq i\leq n).
\end{gathered}
\end{equation}
In particular, $\operatorname{cap}_1=0$, and
\begin{equation}\label{eq:theta-code-capacity-recursion}
\operatorname{cap}_i
=\operatorname{cap}_{i-1}+1-\beta_i-\gamma_i
\qquad(2\leq i\leq n).
\end{equation}
The subset $\operatorname{Code}_{n,k,r}$ is defined by the additional
conditions
\[
\sum_{i=1}^n\beta_i=u-1,
\qquad \sum_{i=1}^n\gamma_i=r.
\]
Thus $\operatorname{cap}_n=k$ for every code in this subset.

Fix $\pi\in\operatorname{Red}_{n,k,r}$, and use its associated
permutation $\sigma$, inverse permutation $\eta=\sigma^{-1}$, and type word $C$ as mentioned at the beginning of Section~\ref{sec:proofs-theta-formulas}.
Define
\begin{equation}\label{eq:theta-code-bits}
\beta_i:=\one_{\sigma_i\in\operatorname{NonExc}(C)\setminus\{1\}},
\qquad
\gamma_i:=\one_{i\in\operatorname{Prim}(C)}.
\end{equation}
The subscript of $\beta_i$ specifies a position in $\sigma$, whereas
the subscript of $\gamma_i$ specifies a value. For $1\leq i\leq n$,
define a set of position indices and a set of values, respectively, by
\begin{align}
\mathcal{P}\mathrm{os}_i
&:=\{a\in[i-1]:\sigma_a\in\operatorname{Bar}(C)
                         \text{ or }\sigma_a\notin[2,i]\},
\label{eq:theta-position-set}\\
\mathcal{V}\mathrm{al}_i
&:=\{v\in[2,i]:v\in\operatorname{Bar}(C)\text{ or }\eta_v>i\}.
\label{eq:theta-available-values}
\end{align}
We take $[0]=[2,1]=\varnothing$, so that
$\mathcal{P}\mathrm{os}_1=\mathcal{V}\mathrm{al}_1=\varnothing$.

We define the remaining coordinate directly by counting elements
of these two sets:
\begin{equation}\label{eq:theta-code-coordinate}
\alpha_i:=
\begin{cases}
\#\{a\in\mathcal{P}\mathrm{os}_i:\pi_a^D<\pi_i^D\},&\beta_i=0,\\[1mm]
\#\{v\in\mathcal{V}\mathrm{al}_i:v>\sigma_i\},&\beta_i=1.
\end{cases}
\end{equation}
The first comparison is in the ordered alphabet of
Definition~\ref{def:DEX}; the second is a comparison of integers.
Write $\operatorname{Enc}(\pi)=((\alpha_i,\beta_i,\gamma_i))_{i=1}^n$.

\begin{lemma}\label{lem:theta-code-ranks}
The sequence $\operatorname{Enc}(\pi)$ belongs to
$\operatorname{Code}_{n,k,r}$. More precisely, for every $i\in[n]$,
\begin{equation}\label{eq:theta-pos-val-cardinalities}
|\mathcal{P}\mathrm{os}_i|=\operatorname{cap}_i+\beta_i,
\qquad |\mathcal{V}\mathrm{al}_i|=\operatorname{cap}_i.
\end{equation}
Furthermore, for $2\leq i\leq n$,
\begin{equation}\label{eq:theta-available-set-transition}
\mathcal{V}\mathrm{al}_i
=\left(\mathcal{V}\mathrm{al}_{i-1}\cup
 \begin{cases}\{i\},&\gamma_i=0,\\
                \varnothing,&\gamma_i=1
 \end{cases}\right)
 \setminus
 \begin{cases}\{\sigma_i\},&\beta_i=1,\\
                \varnothing,&\beta_i=0.
 \end{cases}
\end{equation}
\end{lemma}

\begin{proof}
By \eqref{eq:theta-position-set}, we claim that
\begin{equation}\label{eq:theta-pos-partition}
    [i-1]=\mathcal{P}\mathrm{os}_i\sqcup\{a\in[i-1]\,:\,\beta_a=1\}\sqcup\{\eta_v:v\in\operatorname{Prim}(C)\cap[i]\}.
\end{equation}
In fact, this is because
\begin{align*}
    [i-1]\setminus\mathcal{P}\mathrm{os}_i &= \{a\in[i-1]:\text{$\sigma_a\notin\operatorname{Bar}(C)$ and $\sigma_a\in[2,i]$}\} \\
    &= \{a\in[i-1]:\text{$\sigma_a\notin\operatorname{Bar}(C)$ and $\sigma_a\in[2,i]$ and $\sigma_a\le a$}\}\sqcup \\
    &\quad\{a\in[i-1]:\text{$\sigma_a\notin\operatorname{Bar}(C)$ and $\sigma_a\in[2,i]$ and $\sigma_a>a$}\} \\
    &= \{a\in[i-1]:\text{$\sigma_a\notin\operatorname{Bar}(C)$ and $\sigma_a\in[2,i]$ and $\sigma_a\le a$}\}\sqcup \\
    &\quad\{a\in[n]:\text{$\sigma_a\notin\operatorname{Bar}(C)$ and $\sigma_a\in[2,i]$ and $\sigma_a>a$}\} \\
    &=\{a\in[i-1]:\beta_a=1\}\sqcup\{\eta_v:v\in\operatorname{Prim}(C)\cap[2,i]\}\\
    &=\{a\in[i-1]:\beta_a=1\}\sqcup\{\eta_v:v\in\operatorname{Prim}(C)\cap[i]\}.
\end{align*}
Therefore, Equation~\eqref{eq:theta-pos-partition} gives
\begin{align*}
|\mathcal{P}\mathrm{os}_i|
&=i-1-\sum_{a<i}\beta_a
   -|\operatorname{Prim}(C)\cap[i]|\\
&=i-1-\sum_{a<i}\beta_a-\sum_{a\leq i}\gamma_a
 =\operatorname{cap}_i+\beta_i.
\end{align*}

To count the values in $\mathcal{V}\mathrm{al}_i$, we claim that the following union is disjoint:
\begin{equation}\label{eq:theta-value-partition}
[2,i]
=\mathcal{V}\mathrm{al}_i
 \sqcup\{\sigma_a:a\leq i,\ \beta_a=1\}
 \sqcup\bigl(\operatorname{Prim}(C)\cap[i]\bigr).
\end{equation}
Indeed, we have
\begin{align*}
    [2,i]\setminus\mathcal{V}\mathrm{al}_i &=\{v\in[2,i]:v\notin\operatorname{Bar}(C)\text{ and }\eta_v\le i\} \\
    &=\{v\in[2,i]:v\notin\operatorname{Bar}(C)\text{ and }\eta_v\le i\text{ and }\eta_v\ge v\}\sqcup \\
    &\quad\{v\in[2,i]:v\notin\operatorname{Bar}(C)\text{ and }\eta_v\le i\text{ and }\eta_v< v\} \\
    &=\{v\in[2,n]:v\notin\operatorname{Bar}(C)\text{ and }\eta_v\le i\text{ and }\eta_v\ge v\}\sqcup \\
    &\quad\{v\in[i]:v\notin\operatorname{Bar}(C)\text{ and }\eta_v< v\} \\
    &=\{\sigma_a:a\leq i,\ \beta_a=1\}\sqcup\bigl(\operatorname{Prim}(C)\cap[i]\bigr),
\end{align*}
so Equation~\eqref{eq:theta-value-partition} holds, which gives
\[
|\mathcal{V}\mathrm{al}_i|
=i-1-\sum_{a\leq i}\beta_a-\sum_{a\leq i}\gamma_a
=\operatorname{cap}_i.
\]

If $\beta_i=0$, then $\alpha_i$ counts a subset of
$\mathcal{P}\mathrm{os}_i$, whose cardinality is $\operatorname{cap}_i$.
If $\beta_i=1$, then $\alpha_i$ counts a subset of
$\mathcal{V}\mathrm{al}_i$, with the same cardinality. Thus
$0\leq\alpha_i\leq\operatorname{cap}_i$ in both cases.
If $\sigma_1>1$, its occurrence at position $1$ is either barred
or an unbarred excedance, so $\beta_1=0$; if $\sigma_1=1$, it immediately follows that $\beta_1=0$. Also $1\notin\operatorname{Prim}(C)$ gives
$\gamma_1=0$, and $\mathcal{P}\mathrm{os}_1=\mathcal{V}\mathrm{al}_1=\varnothing$ gives $\alpha_1=0$.
The equalities $\sum_i\beta_i=u-1$ and $\sum_i\gamma_i=r$ follow
from \eqref{eq:theta-code-bits}. We thereby deduce that $\operatorname{Enc}(\pi)\in\operatorname{Code}_{n,k,r}$.

Finally, to show \eqref{eq:theta-available-set-transition}, compare the disjoint partitions
\eqref{eq:theta-value-partition} at $i-1$ and $i$, where $i\geq2$. For the rightmost set in \eqref{eq:theta-value-partition}, we have
\[\operatorname{Prim}(C)\cap[i]=\bigl(\operatorname{Prim}(C)\cap[i-1]\bigr)\cup\begin{cases}
    \varnothing, &\gamma_i=0\\
    \{i\}, &\gamma_i=1.
\end{cases}\]
For the middle set in \eqref{eq:theta-value-partition}, we have
\[\{\sigma_a:a\leq i,\ \beta_a=1\}=\{\sigma_a:a\leq i-1,\ \beta_a=1\}\cup\begin{cases}
    \{\sigma_i\}, &\beta_i=1\\
    \varnothing, &\beta_i=0.
\end{cases}\]
Combining these two identities and \eqref{eq:theta-value-partition} yields \eqref{eq:theta-available-set-transition}.
\end{proof}

We now describe an inverse algorithm $\operatorname{Dec}$ of $\operatorname{Enc}$ for any
$\mathbf c\in\operatorname{Code}_n$. It uses an ordered list of
\emph{position indices} and a set of unused \emph{values}, both
initially empty. A position index in the list specifies a fixed
position of $\sigma$ that has not yet been assigned a value.
The order of indices $i$ in this auxiliary list need not be their
numerical order but aligns with the alphabet order of their values $\pi_i^D$ according to Definition~\ref{def:DEX}.

For $i=1,\ldots,n$, perform the following two operations in the
stated order.
\begin{enumerate}[label=\textup{(\arabic*)}]
\item\label{op:1} If $\gamma_i=1$, remove the leftmost index $a$ from the
position list, set $\sigma_a=i$, and declare this value unbarred.
If $\gamma_i=0$, add the value $i$ to the set of unused values.
\item\label{op:2} If $\beta_i=1$, select the $(\alpha_i+1)$st largest element
of the set of unused values after excluding $1$. Set $\sigma_i$
equal to the selected value, declare it unbarred, and remove it
from the set. If $\beta_i=0$, insert the index $i$ into the position
list so that exactly $\alpha_i$ indices are to its left in that list.
\end{enumerate}
After the last step, assign the value $1$ to the position specified
by the rightmost index in the position list, and declare it unbarred. Assign the other unused
values in increasing order to the positions specified by the other
indices, read from left to right in the list, and bar those values.
The value $1$ remains unbarred. Denote the resulting permutation by
$\operatorname{Dec}(\mathbf c)$.

\begin{example}\label{ex:theta-dec-scan}
\normalfont
We exhibit how to reconstruct the bicolored permutation $\pi=\operatorname{Dec}(\mathbf c)$ from $\mathbf{c}=\operatorname{Enc}(\pi)$ in four diagrams below, where $\mathbf{c}=((\alpha_i,\beta_i,\gamma_i))_{i=1}^n$ is given by the table below. Place positions above values, with both rows labelled $1,\ldots,n$.
An arrow $a\to v$ represents $\sigma_a=v$, with its bar status
recorded separately. At column $i$, process the value first and the
position second. During the scan, an unbarred arrow with target
different from $1$ is assigned exactly when both endpoints are to
the left of the scan line. Among the scanned vertices:
\begin{itemize}[leftmargin=1.6em,itemsep=0pt,topsep=3pt]
\item a scanned position remains unassigned exactly when its arrow
is barred, has target $1$, or has its target to the right of the
scan line;
\item a scanned value other than $1$ remains unused exactly when
its arrow is barred or has its source to the right of the scan line.
\end{itemize}
The value $1$ is reserved until the final assignments. The auxiliary
position list has its own order; it is independent of the fixed order $1,2,\dots,n$ in the upper row in the four diagrams below.

Take $\mathbf c\in\operatorname{Code}_{7,2,2}$ as specified in the
second column below. The last two columns record the state after each
complete step; a dash means no new assignment.
\begin{center}
\small
\setlength{\tabcolsep}{5pt}
\renewcommand{\arraystretch}{1.22}
\begin{tabular}{c|c|c|c|c|c}
$i$ & $(\alpha_i,\beta_i,\gamma_i)$ & $\operatorname{cap}_i$
& \shortstack{New unbarred\\arrows}
& \shortstack{Ordered\\position list}
& \shortstack{Unused\\values}\\ \hline
$1$ & $(0,0,0)$ & $0$ & --- & $(1)$ & $\{1\}$\\
$2$ & $(1,0,0)$ & $1$ & --- & $(1,2)$ & $\{1,2\}$\\
$3$ & $(1,1,0)$ & $1$ & $3\to2$ & $(1,2)$ & $\{1,3\}$\\
$4$ & $(1,0,0)$ & $2$ & --- & $(1,4,2)$ & $\{1,3,4\}$\\
$5$ & $(1,0,1)$ & $2$ & $1\to5$ & $(4,5,2)$ & $\{1,3,4\}$\\
$6$ & $(1,1,1)$ & $1$ & $4\to6,\ 6\to3$ & $(5,2)$ & $\{1,4\}$\\
$7$ & $(1,0,0)$ & $2$ & --- & $(5,7,2)$ & $\{1,4,7\}$
\end{tabular}
\end{center}
For example, $\alpha_4=1$ inserts $4$ between $1$ and $2$ in the
position list. At step $6$, $\gamma_6=1$ removes the leftmost index
$4$, assigning $\sigma_4=6$; then $\beta_6=1$ and $\alpha_6=1$ select
the second largest value in $\{3,4\}$, assigning $\sigma_6=3$.
Finally, the list $(5,7,2)$ gives $\sigma_2=1$; its other two indices
receive $4<7$, both barred. Thus
\[
\operatorname{Dec}(\mathbf c)
=5\,1\,2\,6\,\overline{4}\,3\,\overline{7}.
\]
In the diagrams below, solid arrows are assigned; dashed arrows show the
final matching for reference only, not information known to
$\operatorname{Dec}$. Orange arrows are barred, so $5\to4$ remains
unassigned at step $5$ although both endpoints have been scanned;
violet arrows have target $1$.
The records shown are after complete steps (i.e. both operations \ref{op:1} and \ref{op:2} are finished, which means a complete step): the unused values are
$\{1\}\sqcup\mathcal{V}\mathrm{al}_i$, and the position list contains
$\mathcal{P}\mathrm{os}_i$ together with $i$ if $\beta_i=0$.

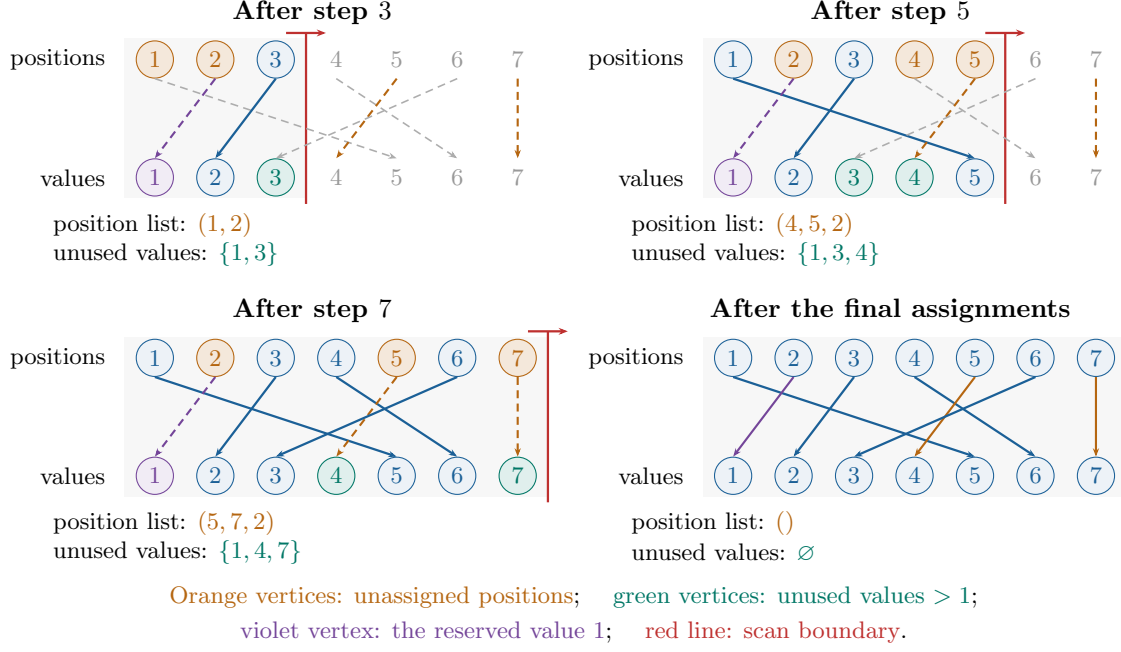
\begingroup
\definecolor{thetaDecBlue}{RGB}{27,96,151}
\definecolor{thetaDecOrange}{RGB}{181,102,17}
\definecolor{thetaDecGreen}{RGB}{0,119,102}
\definecolor{thetaDecViolet}{RGB}{116,69,153}
\definecolor{thetaDecScan}{RGB}{190,48,47}
\newcommand{\thetaDecPanel}[4]{%
\begin{tikzpicture}[x=.80cm,y=.78cm,
  every node/.style={font=\footnotesize},
  theta vertex/.style={circle,inner sep=1.2pt,minimum size=14pt},
  theta future/.style={theta vertex,text=black!40},
  theta assigned/.style={theta vertex,draw=thetaDecBlue,
                         fill=thetaDecBlue!8,text=thetaDecBlue},
  theta position/.style={theta vertex,draw=thetaDecOrange,
                         fill=thetaDecOrange!15,text=thetaDecOrange},
  theta value/.style={theta vertex,draw=thetaDecGreen,
                     fill=thetaDecGreen!12,text=thetaDecGreen},
  theta one/.style={theta vertex,draw=thetaDecViolet,
                   fill=thetaDecViolet!10,text=thetaDecViolet},
  theta done/.style={draw=thetaDecBlue,line width=.85pt,
                    -{Stealth[length=3.2pt]}},
  theta later/.style={draw=black!32,densely dashed,line width=.6pt,
                     -{Stealth[length=3.2pt]}},
  theta barred/.style={draw=thetaDecOrange,densely dashed,
                      line width=.8pt,-{Stealth[length=3.2pt]}},
  theta reserved/.style={draw=thetaDecViolet,densely dashed,
                        line width=.8pt,-{Stealth[length=3.2pt]}}]
\path[use as bounding box] (-1.35,-3.55) rectangle (7.65,.95);
\ifnum#1<8
  \fill[black!3] (.5,.35) rectangle (#1+.5,-2.35);
  \draw[thetaDecScan,line width=.85pt]
       (#1+.5,.45)--(#1+.5,-2.45);
  \draw[thetaDecScan,-{Stealth[length=3.5pt]},line width=.85pt]
       (#1+.17,.45)--(#1+.83,.45);
\else
  \fill[black!3] (.5,.35) rectangle (7.5,-2.35);
\fi
\node[font=\small\bfseries] at (3.6,.8) {#2};
\node[anchor=east] at (.35,0) {positions};
\node[anchor=east] at (.35,-2) {values};
\foreach \thetaSource/\thetaTarget/\thetaType in
 {1/5/0,2/1/2,3/2/0,4/6/0,5/4/1,6/3/0,7/7/1} {
  \def\thetaUpperStyle{theta future}
  \def\thetaLowerStyle{theta future}
  \ifnum#1=8
    \def\thetaUpperStyle{theta assigned}
    \def\thetaLowerStyle{theta assigned}
  \else
    \ifnum\thetaSource>#1\else
      \def\thetaUpperStyle{theta position}
      \ifnum\thetaType=0
        \ifnum\thetaTarget>#1\else
          \def\thetaUpperStyle{theta assigned}
        \fi
      \fi
    \fi
    \ifnum\thetaTarget>#1\else
      \def\thetaLowerStyle{theta value}
      \ifnum\thetaType=0
        \ifnum\thetaSource>#1\else
          \def\thetaLowerStyle{theta assigned}
        \fi
      \fi
      \ifnum\thetaTarget=1
        \def\thetaLowerStyle{theta one}
      \fi
    \fi
  \fi
  \node[\thetaUpperStyle] (thetaUpper\thetaSource)
       at (\thetaSource,0) {$\thetaSource$};
  \node[\thetaLowerStyle] (thetaLower\thetaTarget)
       at (\thetaTarget,-2) {$\thetaTarget$};
}
\foreach \thetaSource/\thetaTarget/\thetaType in
 {1/5/0,2/1/2,3/2/0,4/6/0,5/4/1,6/3/0,7/7/1} {
  \ifnum\thetaType=0
    \def\thetaArrowStyle{theta done}
    \ifnum\thetaSource>#1
      \def\thetaArrowStyle{theta later}
    \fi
    \ifnum\thetaTarget>#1
      \def\thetaArrowStyle{theta later}
    \fi
  \else
    \ifnum\thetaType=1
      \def\thetaArrowStyle{theta barred}
    \else
      \def\thetaArrowStyle{theta reserved}
    \fi
  \fi
  \ifnum#1=8
    \draw[\thetaArrowStyle,solid]
      (thetaUpper\thetaSource.south)--(thetaLower\thetaTarget.north);
  \else
    \draw[\thetaArrowStyle]
      (thetaUpper\thetaSource.south)--(thetaLower\thetaTarget.north);
  \fi
}
\node[anchor=west] at (-.85,-2.8)
  {position list: \textcolor{thetaDecOrange}{$#3$}};
\node[anchor=west] at (-.85,-3.3)
  {unused values: \textcolor{thetaDecGreen}{$#4$}};
\end{tikzpicture}%
}
\begin{center}
\begin{tabular}{@{}c@{\hspace{.45cm}}c@{}}
\thetaDecPanel{3}{After step $3$}{(1,2)}{\{1,3\}}
&
\thetaDecPanel{5}{After step $5$}{(4,5,2)}{\{1,3,4\}}\\[3mm]
\thetaDecPanel{7}{After step $7$}{(5,7,2)}{\{1,4,7\}}
&
\thetaDecPanel{8}{After the final assignments}{()}{\varnothing}
\end{tabular}

\smallskip
{\footnotesize
\textcolor{thetaDecOrange}{Orange vertices: unassigned positions};
\quad
\textcolor{thetaDecGreen}{green vertices: unused values $>1$};\\
\textcolor{thetaDecViolet}{violet vertex: the reserved value $1$};
\quad
\textcolor{thetaDecScan}{red line: scan boundary}.}
\end{center}
\endgroup
\end{example}

\begin{proposition}\label{prop:theta-prefix-bijection}
The algorithm $\operatorname{Dec}$ is well-defined on
$\operatorname{Code}_n$, and its restriction to
$\operatorname{Code}_{n,k,r}$ is the inverse of
$\operatorname{Enc}$. In particular, the following map is a bijection:
\[
\operatorname{Enc}:\operatorname{Red}_{n,k,r}
\longrightarrow\operatorname{Code}_{n,k,r}.
\]
\end{proposition}

\begin{proof}
Note that each step includes two operations \ref{op:1} and \ref{op:2}. Throughout this proof, ``after step $i$'' means after both operations in step $i$ have been finished. 

We first prove that the operations are well-defined. At step $1$,
the algorithm adds the value $1$ to the unused values and inserts
the index $1$ into the empty position list. Hence both have
$\operatorname{cap}_1+1=1$ element. Inductively, suppose that both
have $\operatorname{cap}_{i-1}+1$ elements after step $i-1$. We next analyze step $i$.
If $\gamma_i=1$, the position list is nonempty and the required
deletion is possible. If $\gamma_i=0$, the new value $i$ has not
been used, since every value introduced at an earlier step is at
most $i-1$. After operation \ref{op:1}, the respective cardinalities are
\begin{align*}
\#\{\text{indices in the position list}\}
 &=\operatorname{cap}_{i-1}+1-\gamma_i
  =\operatorname{cap}_i+\beta_i,\\
\#\{\text{unused values}\}
 &=\operatorname{cap}_{i-1}+2-\gamma_i
  =\operatorname{cap}_i+\beta_i+1.
\end{align*}
If $\beta_i=0$, there are $\operatorname{cap}_i+1$ possible places
to insert the index $i$ into the position list. If $\beta_i=1$, the unused values include
$1$, which is never selected in operation \ref{op:2}; excluding it leaves
$\operatorname{cap}_i+1$ choices. Thus
$0\leq\alpha_i\leq\operatorname{cap}_i$ makes operation \ref{op:2}
well-defined in both cases. Afterwards both cardinalities equal
$\operatorname{cap}_i+1$. This completes the induction, and also
proves that the final assignments after the last step are possible, because we always have $\#\{\text{indices in the position list}\}=\#\{\text{unused values}\}=\operatorname{cap}_i+1$ after step $i$.

We then prove that $\operatorname{Dec}$ maps $\operatorname{Code}_{n,k,r}$ into $\operatorname{Red}_{n,k,r}$. Write $\pi=\operatorname{Dec}(\mathbf c)$, and let $C$ be its associated type word given by \eqref{eq:color-word-from-pi}. Every index $a$ removed from the position list in operation \ref{op:1} at step $i$ was inserted at a
step $a<i$, implying that $a\in\EXC{\pi}$. Every value $v$ selected in operation \ref{op:2} belongs to $[2,i]$,
so $v\in\operatorname{NonExc}(C)$. Besides these selected values, we also have $1\in\operatorname{NonExc}(C)$.
An index $i$ is inserted into the position list only at step $i$ and only if that position
is not assigned a value in operation \ref{op:2}. All subsequent assignments
at that position remove its index from the list, or occur in the
final assignments. Thus every position is assigned to a value exactly once.
Likewise, each value is assigned to a position exactly once, and is either
assigned immediately in operation \ref{op:1}, removed from the unused values
in operation \ref{op:2}, or assigned at the end. The result $\sigma$ is a permutation, which gives $\pi=\operatorname{Dec}(\mathbf c)$ after bars are added.
It follows that
\[
\exc{\pi}=\sum_{i=1}^n\gamma_i,
\qquad |\operatorname{NonExc}(C)|=1+\sum_{i=1}^n\beta_i,
\qquad \numcolor{\pi}=n-\sum_{i=1}^n\gamma_i-\sum_{i=1}^n\beta_i-1.
\]
In particular, $\operatorname{Dec}$ maps
$\operatorname{Code}_{n,k,r}$ into $\operatorname{Red}_{n,k,r}$.
For the rest of the proof, restrict $\mathbf c$ to
$\operatorname{Code}_{n,k,r}$.

We next prove $\operatorname{Enc}(\operatorname{Dec}(\mathbf c)) = \mathbf c$. Write $\mathbf c = ((\alpha_i,\beta_i,\gamma_i))_{i=1}^n$.
For the resulting permutation $\pi=\operatorname{Dec}(\mathbf c)$, the position list always has
the following property: if index $a$ is to the left of index $b$
in the list, then $\pi_a^D<\pi_b^D$. To prove this, fix two indices
that are simultaneously in the list. Their relative order cannot
change after both are inserted into the list. If the right index is eventually removed
in operation \ref{op:1}, then the left index must be removed at an earlier
step. The two assigned values therefore give primed letters in
strictly increasing order. If only the left index is removed in
operation \ref{op:1}, its primed letter is smaller than the barred letter
or ordinary $1$ eventually assigned to the right index. If neither
is removed in operation \ref{op:1}, the final assignments give increasing
barred letters, with the unbarred $1$ assigned at the rightmost
index in the list. Since every barred letter is smaller than the unbarred $1$,
the claimed inequality holds in this case as well.

Just after operation \ref{op:1} of step $i$, an index $a<i$ is still in
the position list precisely when $\pi_a$ is barred or $\sigma_a\notin[2,i]$.
Indeed, if $a<i$ and $\pi_a=\sigma_a\in[2,i]$ is unbarred, then there are two cases: either $\sigma_a\le a$ and hence $a$ was assigned by operation \ref{op:2} of step $a$, or $\sigma_a>a$ and hence $a$ was removed from the list by operation \ref{op:1} of step $\sigma_a$. Hence
\eqref{eq:theta-position-set} identifies the indices in the list
with $\mathcal{P}\mathrm{os}_i$.
Thus, when $\beta_i=0$, inserting $i$ with $\alpha_i$ indices on
its left makes $\alpha_i$ equal to the first count in
\eqref{eq:theta-code-coordinate} (recall the order preserving property shown in the previous paragraph). After completing step $i$, the
unused values are precisely
\begin{equation}\label{eq:theta-decoder-unused-values}
\{1\}\sqcup\mathcal{V}\mathrm{al}_i.
\end{equation}
Indeed, among $[2,i]$ the values already assigned are exactly
$\operatorname{Prim}(C)\cap[i]$ and the values assigned
at positions $a\leq i$ with $\beta_a=1$, so applying
\eqref{eq:theta-value-partition} gives \eqref{eq:theta-decoder-unused-values}. If $\beta_i=1$, the value selected
in operation \ref{op:2} has exactly $\alpha_i$ larger values in the unused
set. Removing the selected value does not change that count, and
the reserved value $1$ is smaller than the selected value. Therefore
\eqref{eq:theta-decoder-unused-values} identifies this count with
the second count in \eqref{eq:theta-code-coordinate}. The type
verification in the third paragraph of this proof also recovers $\beta_i$ and $\gamma_i$.
Hence $\operatorname{Enc}(\operatorname{Dec}(\mathbf c))=\mathbf c$.

Conversely, we prove $\operatorname{Dec}(\operatorname{Enc}(\pi))=\pi$. Start with a given $\pi\in\operatorname{Red}_{n,k,r}$
and apply the decoding algorithm $\operatorname{Dec}$ to its code $\operatorname{Enc}(\pi)=((\alpha_i,\beta_i,\gamma_i))_{i=1}^n$. Just after step $1$, the
position list is $(1)$ and the set of unused values is $\{1\}$, as required.
Inductively, given $i>1$, assume that all assignments at all steps $j<i$
agree with $\pi$, that the position list is ordered by increasing
$\pi_a^D$, and that its set of unused values equals
$\{1\}\sqcup\mathcal{V}\mathrm{al}_{j}$. We then analyze step $i$.

We first process operation \ref{op:1} of step $i$. If $\gamma_i=1$, then
$\eta_i<i$, so the index $\eta_i$ is in the position list.
Its letter $\pi_{\eta_i}^D=i^\prime$ is smaller than every other letter $\pi_a^D$ associated to
an index $a\neq\eta_i$ in that list: for any other index $a\neq\eta_i$ in the list, either $\pi_a^D$ has no prime and hence $\pi_a^D>i^\prime$ according to the alphabet order, or $\pi_a^D=\sigma_a^\prime$ is primed and thus $\sigma_a>i$ (otherwise $\sigma_a<i$, so $a$ was removed from the position list by operation \ref{op:1} at step $\sigma_a<i$, which contradicts the fact that $a$ is still in the list) implies $\pi_a^D>i^\prime$. Consequently, the induction hypothesis means that $\eta_i$ is leftmost in the position list. Thus operation \ref{op:1} removes exactly $\eta_i$
and makes the correct assignment. If $\gamma_i=0$, it adds the
value $i$, as in \eqref{eq:theta-available-set-transition}.
Using the induction hypothesis, the same description of the
unassigned positions shows that after operation \ref{op:1} the position
list consists of $\mathcal{P}\mathrm{os}_i$ in the stated order.

We second process operation \ref{op:2} of step $i$.
If $\beta_i=0$, the first
count in \eqref{eq:theta-code-coordinate} inserts $i$ in exactly
the place that preserves the position list order mentioned in the induction hypothesis. If $\beta_i=1$, the unused
values other than $1$ are
$\mathcal{V}\mathrm{al}_i\sqcup\{\sigma_i\}$, by
\eqref{eq:theta-available-set-transition}. The second count in
\eqref{eq:theta-code-coordinate} therefore selects exactly
$\sigma_i$. This proves the induction. 

Finally (i.e. after step $n$), the induction result implies that the unused values
are $\{1\}\sqcup\operatorname{Bar}(C)$, and the remaining indices
are ordered by their increasing letters in $\pi^D$. The prescribed
final assignments consequently recover the remaining entries of
$\pi$. Thus $\operatorname{Dec}(\operatorname{Enc}(\pi))=\pi$.
\end{proof}

\subsection{The crossing weight and the descent set of the code}
\label{subsec:theta-code-statistics}

The bijectivity of $\operatorname{Enc}$ has now been established.
We turn to figuring out what the crossing weight and the descent set become after the map $\pi\mapsto\operatorname{Enc}(\pi)=((\alpha_i,\beta_i,\gamma_i))_{i=1}^n$, beginning with
an alternative expression for the coordinates $\alpha_i$.
For this purpose, we first record the set identity
\begin{equation}\label{eq:theta-position-set-difference}
\mathcal{P}\mathrm{os}_i
=\{a<i:\beta_a=0\}
 \setminus\{\eta_v:v\in\operatorname{Prim}(C)\cap[i]\}.
\end{equation}
Indeed, the positions with $\beta_a=0$ have primed letters, barred
letters, or the ordinary letter $1$. For every
$v\in\operatorname{Prim}(C)\cap[i]$, admissibility gives
$\eta_v<v\leq i$, so $\eta_v<i$ and $\beta_{\eta_v}=0$. Thus
the set removed on the right is a subset of $\{a<i:\beta_a=0\}$,
and its removal leaves exactly the positions with barred values,
the ordinary value $1$, or primed values greater than $i$.
If $a<i$ and $\sigma_a>i$ is unbarred, then $a<\sigma_a$, so
$\sigma_a$ is primed. The remaining positions are therefore exactly
those in \eqref{eq:theta-position-set}, proving the identity.

We introduce the two auxiliary counts
\begin{equation}\label{eq:theta-prefix-counts}
L_i:=\#\{a<i:\pi_a^D<\pi_i^D\},
\qquad \nu_i:=|\operatorname{Prim}(C)\cap[i]|.
\end{equation}
The definition \eqref{eq:theta-code-coordinate} implies the following new counting of $\alpha_i$:
\begin{lemma}\label{lem:theta-code-algebraic-coordinate}
For $i\in[n]$, we have
\begin{equation}\label{eq:theta-code-algebraic-coordinate}
\alpha_i=
\begin{cases}
L_i-\nu_i,&\beta_i=0,\\[1mm]
L_i-(\sigma_i-1)-|\operatorname{Prim}(C)\cap(\sigma_i,i]|,
 &\beta_i=1.
\end{cases}
\end{equation}
\end{lemma}
\begin{proof}
    Suppose first that $\beta_i=0$.
For every $v\in\operatorname{Prim}(C)\cap[i]$, admissibility gives
$\eta_v<v\leq i$. We also have $v'<\pi_i^D$: this is immediate
if $\pi_i^D$ is barred or is the ordinary $1$; if
$\pi_i^D=\sigma_i'$, then $i<\sigma_i$, and hence
$v\leq i<\sigma_i$ gives $v'<\sigma_i'=\pi_i^D$.
Furthermore, every letter of $\pi^D$ at a position with $\beta_a=1$ is greater
than $\pi_i^D$ because $\pi_i^D$ is primed, is barred, or is the ordinary $1$. Thus
\[
\{\eta_v:v\in\operatorname{Prim}(C)\cap[i]\}
\subseteq\{a<i:\pi_a^D<\pi_i^D\}
\subseteq\{a<i:\beta_a=0\}.
\]
Using this set chain, \eqref{eq:theta-position-set-difference}, and the injectivity
of $\eta$, we obtain
\begin{align*}
\alpha_i
&=\#\{a\in\mathcal{P}\mathrm{os}_i:\pi_a^D<\pi_i^D\}\\
&=\#\{a<i:\pi_a^D<\pi_i^D\}
  -\#\{\eta_v:v\in\operatorname{Prim}(C)\cap[i]\}
 =L_i-\nu_i.
\end{align*}

Suppose next that $\beta_i=1$, and put $y=\sigma_i$.
Then $2\leq y\leq i$ and $\pi_i^D=y$. For $a<i$,
\[
\pi_a^D>\pi_i^D
\quad\Longleftrightarrow\quad
\beta_a=1\text{ and }\sigma_a>y,
\]
because the primed letters, barred letters, and ordinary $1$ are
all smaller than $y$. Since the letters are distinct,
\[
L_i=i-1-\#\{a<i:\pi_a^D>\pi_i^D\}=i-1-\#\{a<i:\beta_a=1,\ \sigma_a>y\}.
\]
Intersecting the disjoint partition \eqref{eq:theta-value-partition}
with $(y,i]$ gives
\[i-y=\alpha_i+\#\{\sigma_a:a\leq i,\ \beta_a=1,\sigma_a>y\}
 +\lvert\operatorname{Prim}(C)\cap(y,i]\rvert\]
and hence
\begin{align*}
\alpha_i
&=(i-y)-|\operatorname{Prim}(C)\cap(y,i]|
         -\#\{a\le i:\beta_a=1,\ \sigma_a>y\}\\
&=(i-y)-|\operatorname{Prim}(C)\cap(y,i]|
         -\#\{a<i:\beta_a=1,\ \sigma_a>y\}=L_i-(y-1)-|\operatorname{Prim}(C)\cap(y,i]|.
\end{align*}
The position $a=i$ does not contribute to the removed set, since
$\sigma_i=y$. This proves \eqref{eq:theta-code-algebraic-coordinate}.
\end{proof}

\begin{lemma}\label{lem:theta-code-crossing}
For every $\pi\in\operatorname{Red}_{n,k,r}$,
\[
\sum_{i=1}^n\alpha_i=\cro{\pi}.
\]
\end{lemma}

\begin{proof}
The standardization of $\pi^D$ is
$\operatorname{PosWord}(\eta,C)^{-1}$, as shown in the proof of
Proposition~\ref{prop:barred-unbarred-relation}. Standardization
preserves inversions, and a permutation and its inverse have equal
inversion numbers. Every pair of positions contributes either an
inversion or a contribution to exactly one $L_i$. Therefore
\begin{equation}\label{eq:theta-sum-prefix-ranks}
\sum_{i=1}^n L_i=\binom n2-\inv\bigl(\operatorname{PosWord}(\eta,C)\bigr).
\end{equation}
Equation~\eqref{eq:theta-code-algebraic-coordinate} can be written
uniformly as
\[
\alpha_i=L_i-\nu_i
 -\one_{\sigma_i\in\operatorname{NonExc}(C)}\cdot(\sigma_i-1)
 +\one_{\beta_i=1}\cdot\nu_{\sigma_i}.
\]
For $\beta_i=1$, this uses
$|\operatorname{Prim}(C)\cap(\sigma_i,i]|=\nu_i-\nu_{\sigma_i}$.
At the position $\eta_1$ of the ordinary $1$, it uses $\sigma_{\eta_1}-1=0,\beta_{\eta_1}=0$.
For a barred value $x$, the number $\nu_x$ counts the primed values
strictly smaller than $x$, since $x\notin\operatorname{Prim}(C)$.
We have
\begin{align}
\sum_{i=1}^n\nu_i-\sum_{x\in\operatorname{NonExc}(C)}\nu_x
&=\sum_{x\in\operatorname{Prim}(C)\cup\operatorname{Bar}(C)}
              |\operatorname{Prim}(C)\cap[x]|\notag\\
&=\binom{r+1}{2}+\sum_{x\in\operatorname{Bar}(C)}\nu_x.
\label{eq:theta-prim-count-sum}
\end{align}
Indeed, the terms indexed by the primed values sum to
$1+\cdots+r$, and the terms indexed by the barred values count
exactly the increasing pairs with a primed smaller value and a barred
larger value. Summing the uniform formula
for $\alpha_i$ above consequently gives
\begin{align}
\sum_{i=1}^n\alpha_i
=&\sum_{i=1}^n L_i - \sum_{i=1}^n \nu_i -\sum_{x\in\operatorname{NonExc}(C)}(x-1) +\sum_{x\in\operatorname{NonExc}(C)\setminus\{1\}} \nu_x \notag \\
=&\sum_{i=1}^n L_i - \sum_{i=1}^n \nu_i -\sum_{x\in\operatorname{NonExc}(C)}(x-1) +\sum_{x\in\operatorname{NonExc}(C)} \nu_x \notag \\
=&\binom n2-\inv\bigl(\operatorname{PosWord}(\eta,C)\bigr)
 -\sum_{x\in\operatorname{NonExc}(C)}(x-1)
-\binom{r+1}{2}
 -\sum_{x\in\operatorname{Bar}(C)}\nu_x
\label{eq:theta-alpha-sum-expanded}
\end{align}
where the last equality follows from \eqref{eq:theta-sum-prefix-ranks} and \eqref{eq:theta-prim-count-sum}.

The increasing pairs (i.e. non-inversions) of entries in $\operatorname{ValWord}(C)$ (defined in Proposition~\ref{prop:signed-formulas}) split into four disjoint
classes: pairs within the primed block, pairs within the barred
block, increasing pairs from the primed block to the barred block,
and increasing pairs whose larger value belongs to the nonexcedance
block. For each $x\in\operatorname{NonExc}(C)$, all $x-1$ smaller
values occur to its left in this word, by the definition of its
three increasing blocks. Hence
\begin{align}
\binom n2-\inv\bigl(\operatorname{ValWord}(C)\bigr)
={}&\binom r2+\binom k2
 +\sum_{x\in\operatorname{Bar}(C)}\nu_x\notag\\
&\hspace{8mm}+\sum_{x\in\operatorname{NonExc}(C)}(x-1).
\label{eq:theta-valword-noninversions}
\end{align}
Substituting this into \eqref{eq:theta-alpha-sum-expanded} and using
$\binom r2-\binom{r+1}{2}=-r$ yields
\[
\sum_{i=1}^n\alpha_i
=\binom k2-r
 +\inv\bigl(\operatorname{ValWord}(C)\bigr)
 -\inv\bigl(\operatorname{PosWord}(\eta,C)\bigr).
\]
This equals $\cro{\pi}$ by \eqref{eq:signed-cr} and
Definition~\ref{def:cr}.
\end{proof}

For $\mathbf c\in\operatorname{Code}_n$, define
$\operatorname{Asc}(\mathbf c)\subseteq[n-1]$ by
\begin{equation}\label{eq:theta-code-ascent-set}
\begin{split}
i\in\operatorname{Asc}(\mathbf c)\quad\Longleftrightarrow\quad
&\beta_i<\beta_{i+1},\quad\text{or}\\
&\beta_i=\beta_{i+1}=1\text{ and }
 \alpha_i\geq\alpha_{i+1}+\gamma_{i+1},\quad\text{or}\\
&\beta_i=\beta_{i+1}=0\text{ and }
 \alpha_i<\alpha_{i+1}+\gamma_{i+1}.
\end{split}
\end{equation}
This definition aligns with the definition above \cite[Conjecture 4.1]{Lentfer2025}.

\begin{lemma}\label{lem:theta-code-dex}
For every $\pi\in\operatorname{Red}_{n,k,r}$,
\[
\DEX{\pi}
=[n-1]\setminus\operatorname{Asc}(\operatorname{Enc}(\pi)).
\]
\end{lemma}

\begin{proof}
Fix $i\in[n-1]$. If $(\beta_i,\beta_{i+1})=(0,1)$, then
$\pi_i^D<\pi_{i+1}^D$, since every primed letter, barred letter,
and the ordinary $1$ is smaller than every ordinary letter at least
$2$. Thus $i\notin\DEX{\pi}$ and
$i\in\operatorname{Asc}(\operatorname{Enc}(\pi))$.
If $(\beta_i,\beta_{i+1})=(1,0)$, the same comparison gives
$\pi_i^D>\pi_{i+1}^D$, and the two memberships are reversed.

Suppose $\beta_i=\beta_{i+1}=0$. Then
$\nu_{i+1}=\nu_i+\gamma_{i+1}$ and
\eqref{eq:theta-code-algebraic-coordinate} give
\[
\alpha_i=L_i-\nu_i,
\qquad \alpha_{i+1}+\gamma_{i+1}=L_{i+1}-\nu_i.
\]
If $\pi_i^D<\pi_{i+1}^D$, then
\[
\{a<i:\pi_a^D<\pi_i^D\}\sqcup\{i\}
\subseteq\{a<i+1:\pi_a^D<\pi_{i+1}^D\},
\]
so $L_i+1\leq L_{i+1}$. If $\pi_i^D>\pi_{i+1}^D$, then
\[
\{a<i+1:\pi_a^D<\pi_{i+1}^D\}
\subseteq\{a<i:\pi_a^D<\pi_i^D\},
\]
so $L_{i+1}\leq L_i$. The letters are distinct, and hence
\[
i\in\DEX{\pi}
\quad\Longleftrightarrow\quad L_{i+1}\le L_i \quad\Longleftrightarrow\quad
\alpha_i\geq\alpha_{i+1}+\gamma_{i+1}.
\]
This is the required complement relation in this case.

Suppose $\beta_i=\beta_{i+1}=1$, and put
$y=\sigma_i$, $x=\sigma_{i+1}$. We have $y\leq i$ and $x\leq i+1$.
If $x=i+1$, then $y<x$, $\gamma_{i+1}=0$, and $\alpha_{i+1}=0$:
all values in $\mathcal{V}\mathrm{al}_{i+1}$ are at most $i+1=\sigma_{i+1}$. Thus both
$y>x$ and $\alpha_i<\alpha_{i+1}+\gamma_{i+1}$ are false.

Otherwise $x\leq i$. Since $x$ is an ordinary value other than $1$
and $\eta_x=i+1$, we have $x\in\mathcal{V}\mathrm{al}_i$. Since
$\eta_y=i$, we have $y\notin\mathcal{V}\mathrm{al}_i$. The definition of
$\alpha_i$ and the recurrence
\eqref{eq:theta-available-set-transition} give
\begin{align*}
\alpha_i&=\#\{v\in\mathcal{V}\mathrm{al}_i:v>y\},\\
\alpha_{i+1}+\gamma_{i+1}
 &=1+\#\{v\in\mathcal{V}\mathrm{al}_i:v>x\}
\end{align*}
where the second equality holds because deleting $x$ does not change the count of
values greater than $x$, and the new value $i+1$ is greater than $x$
and is added by the recurrence \eqref{eq:theta-available-set-transition} exactly when $\gamma_{i+1}=0$.
If $y>x$, then
\[
\{v\in\mathcal{V}\mathrm{al}_i:v>y\}
\subseteq\{v\in\mathcal{V}\mathrm{al}_i:v>x\},
\]
which implies $\alpha_i<\alpha_{i+1}+\gamma_{i+1}$.
If $y<x$, then
\[
\{x\}\sqcup\{v\in\mathcal{V}\mathrm{al}_i:v>x\}
\subseteq\{v\in\mathcal{V}\mathrm{al}_i:v>y\},
\]
which implies $\alpha_i\geq\alpha_{i+1}+\gamma_{i+1}$.
As $x\ne y$, we have proved
\[
i\in\DEX{\pi}
\quad\Longleftrightarrow\quad y>x
\quad\Longleftrightarrow\quad
\alpha_i<\alpha_{i+1}+\gamma_{i+1}.
\]
This is the complement of the second case in
\eqref{eq:theta-code-ascent-set}, and completes the proof.
\end{proof}

\subsection{Insertion of the code into a segmented permutation}
\label{subsec:theta-code-insertion}

We use the insertion underlying Lentfer's bijection
\cite[Theorem~6.2, Corollary~6.3, and Proposition~6.5]{Lentfer2025}.
We describe the algorithm and recall all its properties needed here.
Given $\mathbf c\in\operatorname{Code}_n$, start with the single
segment $[1]$. For $i=2,\ldots,n$, insert the new largest value $i$
by the following rule, according to $(\beta_i,\gamma_i)$:
\begin{enumerate}[label={},leftmargin=3.8em]
\item[$(0,0)$] Insert the singleton segment $[i]$ with exactly
$\alpha_i$ segments to its right.
\item[$(1,0)$] Append $i$ to the segment having exactly $\alpha_i$
segments to its right.
\item[$(0,1)$] Prepend $i$ to the segment having exactly $\alpha_i$
segments to its right.
\item[$(1,1)$] Replace two adjacent segments $U\mid V$ by
$U\,i\,V$, choosing the pair so that the resulting segment has
exactly $\alpha_i$ segments to its right.
\end{enumerate}
Denote the resulting segmented permutation by
$\operatorname{Ins}(\mathbf c)$.

\begin{lemma}\label{lem:theta-insertion-bijection}
The map $\operatorname{Ins}$ is a bijection from
$\operatorname{Code}_n$ to all segmented permutations of $[n]$.
It satisfies
\begin{align}
\operatorname{asc}(\operatorname{Ins}(\mathbf c))&=\sum_i\beta_i,
&\operatorname{des}(\operatorname{Ins}(\mathbf c))&=\sum_i\gamma_i,
\label{eq:theta-insertion-shape}\\
\operatorname{sminv}(\operatorname{Ins}(\mathbf c))&=\sum_i\alpha_i,
&\operatorname{Split}(\operatorname{Ins}(\mathbf c))
 &=\operatorname{Asc}(\mathbf c).
\label{eq:theta-insertion-statistics}
\end{align}
\end{lemma}
\begin{proof}
    Given a code $\mathbf c=((\alpha_i,\beta_i,\gamma_i))_{i=1}^n\in\operatorname{Code}_n$, we identify it with the monomial \[\Bigg(\prod_{i=1}^nx_i^{\alpha_i}\Bigg)\Bigg(\prod_{i=1}^n \theta_i^{\beta_i}\Bigg)\Bigg(\prod_{i=1}^n\xi_i^{\gamma_i}\Bigg)\] introduced by \cite[Definition 3.1]{Lentfer2025}. Therefore, the bijectivity and the first three weight-preserving identities immediately follow from \cite[Theorem 6.2]{Lentfer2025}, and the last set identity immediately arises from \cite[Proposition 6.5]{Lentfer2025}.
\end{proof}

\begin{proposition}\label{prop:theta-pointwise-bijection}
The explicitly defined map
\[
\mathcal J:=\operatorname{Ins}\circ\operatorname{Enc}:
\operatorname{Red}_{n,k,r}
\longrightarrow\operatorname{SW}(1^n;u-1,r)
\]
is a bijection, and for every source object $\pi$ it satisfies
\begin{align}
\operatorname{sminv}(\mathcal J(\pi))&=\cro{\pi},
\label{eq:theta-pointwise-crossing}\\
[n-1]\setminus\operatorname{Split}(\mathcal J(\pi))&=\DEX{\pi}.
\label{eq:theta-pointwise-dex}
\end{align}
Consequently,
\begin{align}
\al{\pi}
 &=\delta_{n,k,r}-\operatorname{sminv}(\mathcal J(\pi)),
\label{eq:theta-pointwise-al}\\
\ell(f_\pi)
 &=ru+\delta_{n,k,r}-\operatorname{sminv}(\mathcal J(\pi)).
\label{eq:theta-pointwise-affine-length}
\end{align}
\end{proposition}

\begin{proof}
Proposition~\ref{prop:theta-prefix-bijection}, the bijectivity in Lemma~\ref{lem:theta-insertion-bijection}, and
\eqref{eq:theta-insertion-shape} show that the composite is a
bijection with the stated target. Lemma~\ref{lem:theta-code-crossing}
and the first equality in \eqref{eq:theta-insertion-statistics} prove
\eqref{eq:theta-pointwise-crossing}. Lemma~\ref{lem:theta-code-dex}
and the second equality in \eqref{eq:theta-insertion-statistics} prove
\eqref{eq:theta-pointwise-dex}. Finally,
Corollary~\ref{cor:cr-al-complement-for-pi} gives
\eqref{eq:theta-pointwise-al}, and
Definition~\ref{def:al} gives \eqref{eq:theta-pointwise-affine-length}.
\end{proof}

\subsection{Completion of the two proofs}
\label{subsec:theta-completion}

\begin{proof}[Proof of Theorem~\ref{conj:unbarred-one-theta}]
Proposition~\ref{prop:theta-pointwise-bijection} gives the identity
of quasisymmetric functions
\begin{align}
\sum_{\pi\in\operatorname{Red}_{n,k,r}}
 z^{\cro{\pi}}\FF_{\DEX{\pi},n}(\xxx)
&=\sum_{w\in\operatorname{SW}(1^n;u-1,r)}
 z^{\operatorname{sminv}(w)}
 \FF_{[n-1]\setminus\operatorname{Split}(w),n}(\xxx)\notag\\
&=\Omega_n\left(
 \sum_{w\in\operatorname{SW}(1^n;u-1,r)}
 z^{\operatorname{sminv}(w)}
 \FF_{\operatorname{Split}(w),n}(\xxx)\right)\notag\\
&=\Omega_n\left(
 \left.\Theta_{e_{u-1}}\Theta_{e_{r}}
             \widetilde H_{(k+1)}\right|_{t=0,\,q=z}
             \right)
\label{eq:theta-crossing-fundamental-identity}
\end{align}
where the last equality follows from \eqref{eq:theta-known-smirnov-expansion}.
The argument of $\Omega_n$ is thereby symmetric. Its Macdonald input has degree
$n-(u-1)-r=k+1\geq1$. Hence $\Omega_n$ may be replaced by $\omega$
on that argument, and \eqref{eq:theta-operator-commutativity} gives
\begin{equation}\label{eq:theta-reduced-crossing-expansion}
\sum_{\pi\in\operatorname{Red}_{n,k,r}}
 z^{\cro{\pi}}\FF_{\DEX{\pi},n}(\xxx)=\omega\!\left(
 \left.\Theta_{e_{u-1}}\Theta_{e_{r}}
             \widetilde H_{(k+1)}\right|_{t=0,\,q=z}
             \right)
=\omega\!\left(
 \left.\Theta_{e_r}\Theta_{e_{u-1}}
             \widetilde H_{(k+1)}\right|_{t=0,\,q=z}
             \right).
\end{equation}

By Corollary~\ref{cor:cr-al-complement-for-pi}, every source object
satisfies $\al{\pi}=\delta_{n,k,r}-\cro{\pi}$.
Substituting $z=q^{-1}$ in \eqref{eq:theta-reduced-crossing-expansion}
and multiplying by $q^{\delta_{n,k,r}}$ therefore gives
\begin{align*}
\Phi^{\mathrm{unb}}_{n,k,r}(\xxx;q)
&=q^{\delta_{n,k,r}}
  \sum_{\pi\in\operatorname{Red}_{n,k,r}}
   (q^{-1})^{\cro{\pi}}\FF_{\DEX{\pi},n}(\xxx)\\
&=q^{\delta_{n,k,r}}\omega\!\left(
 \left.\bigl(\Theta_{e_r}\Theta_{e_{u-1}}
                  \widetilde H_{(k+1)}\bigr)
       \right|_{t=0,\,q=q^{-1}}\right).\qedhere
\end{align*}
\end{proof}

\begin{proof}[Proof of Theorem~\ref{conj:bicolored-theta}]
Suppose first that $k\geq1$. The pair of parameters $(k-1,r)$ has
$u+1\geq1$ nonexcedance values, so the reduced theorem just proved
always applies to it, including when $u=0$. The change in the
normalizing exponent is
\begin{align*}
\delta_{n,k-1,r}
&=\binom n2-\binom{r+1}{2}-\binom{u+1}{2}\\
&=\delta_{n,k,r}-u.
\end{align*}
Using Proposition~\ref{prop:barred-unbarred-relation} and Theorem~\ref{conj:unbarred-one-theta} with bar parameter $k-1$ gives
\begin{align*}
\Phi^{\mathrm{bar}}_{n,k,r}(\xxx;q)
&=q^u\Phi^{\mathrm{unb}}_{n,k-1,r}(\xxx;q)\\
&=q^{u+\delta_{n,k-1,r}}\omega\!\left(
 \left.\bigl(\Theta_{e_r}\Theta_{e_u}\widetilde H_{(k)}\bigr)
       \right|_{t=0,\,q=q^{-1}}\right)\\
&=q^{\delta_{n,k,r}}\omega\!\left(
 \left.\bigl(\Theta_{e_r}\Theta_{e_u}\widetilde H_{(k)}\bigr)
       \right|_{t=0,\,q=q^{-1}}\right).
\end{align*}
This is the second summand of
\eqref{eq:bicolored-theta-conjecture}. If $u\geq1$, applying the
reduced theorem also to $(k,r)$ gives the first summand, and
\eqref{eq:one-bar-status-decomposition} proves the result.
If $u=0$, the unbarred-one part is empty (otherwise the value $1$ would
give an unbarred nonexcedance and thus $u>0$). Thus only the second summand is present,
again as stated in the current theorem.

If $k=0$ and $u\geq1$, the barred-one part is empty and Theorem~\ref{conj:unbarred-one-theta} is exactly Theorem~\ref{conj:bicolored-theta}; the second
summand is omitted. Finally, if $k=u=0$, then $r=n$ and the indexing
set is empty: the value $1$ can be neither barred nor an unbarred
excedance. Both summands on the right are omitted, so both sides are
zero. These cases exhaust the parameters in the theorem.
\end{proof}

\section{crossings and alignments in the literature}
\label{sec:comparison-with-signed-statistics}


A different notion of alignments and crossings for a bicolored permutation was introduced by
Corteel, Josuat-Verg\`es, and Kim in 
\cite{CorteelJosuatVergesKim2013}. 
Cho and Park \cite{ChoPark2015} later introduced another version of alignments and studied its relationship with the crossing number defined by Corteel et~al.
Both papers defined the crossing and the alignment of a bicolored permutation under the \emph{sign order} \eqref{eq:SignOrder}.

In this section, we study the relationship between the alignment and crossing statistics introduced in Section \ref{sec:crossings-and-affine-length}, which are defined using the \emph{color order} \eqref{eq:CoinvConjecture}, and the crossing statistic of Corteel et al. and the alignment statistic of Cho and Park.
Throughout this section, when working with the sign order, we identify $\overline{i}$ with $-i$ for each $i\in[n]$.

For $\tau\in \fB_n$ with $[n]\cup[\overline{n}]$ under the sign order \eqref{eq:SignOrder}, Corteel et al. defined \cite[p. 194]{CorteelJosuatVergesKim2013} a crossing of $\tau$ to be an
ordered pair
$(i,j)\in [n]^2$ satisfying one of the following three conditions:
\begin{align}
i<j\leq\tau_i<\tau_j,
&\label{eq:comparison-crossing-upper}\\
-i<j\leq-\tau_i<\tau_j,
&\label{eq:comparison-crossing-mixed}\\
i>j>\tau_i>\tau_j.
&\label{eq:comparison-crossing-lower}
\end{align}
These three conditions are mutually exclusive. 
Denote $\crSgn(\tau)$ the number of such order pairs.

Denote $\alSgn(\tau)$ the number of Cho and Park's alignments. 
Here we will not define their alignment; the reader who is interested in knowing can see \cite[Definition 2.2]{ChoPark2015}. 
We make use of the following relationship between $\crSgn$ and $\alSgn$ found by Cho and Park \cite[Theorem 3.3]{ChoPark2015}: for $\tau\in\fB_n$,
\begin{equation} \label{eq:comparison-CP-complement}
    \alSgn(\tau)+\crSgn(\tau)=(n-\wex(\tau))\left(\wex(\tau)-1+\nega{\tau}\right)-\binom{\nega{\tau}}{2}.
\end{equation}
where $\wex(\tau)\coloneqq\#\{i\in [n]~:~\tau_i\ge i\}$ is the weak excedance number of $\tau$.

\begin{theorem}\label{thm:comparison-signed-statistics}
Fix integers $n\geq1$, $0\leq k\leq n$, and
$0\leq r\leq n-k$, and put $u:=n-k-r$.
There is an explicit bijection
\[
\mathcal T:
\{\pi\in\coloredperm{n}{k}:\exc{\pi}=r\}
\longrightarrow
\{\tau\in \coloredperm{n}{k}: 
  \wex(\tau)=u\}
\]
such that, for every $\pi$ in its domain,
\begin{align}
\cro{\pi}
&=\crSgn(\mathcal T(\pi)),
\label{eq:comparison-crossing-preserved}\\
\al{\pi}
&=\alSgn(\mathcal T(\pi)).
\label{eq:comparison-alignment-preserved}
\end{align}
Moreover, the value $1$ is unbarred in $\pi$ if and only if
$\mathcal T(\pi)_1>0$.
\end{theorem}

\begin{proof}
Let $(\eta,C)\in\operatorname{Adm}_{n,k,r}$ be the admissible
data associated with $\pi$ under
Proposition~\ref{prop:bicolored-admissible-bijection}.
For this proof, abbreviate
\[
P:=\operatorname{Prim}(C),\qquad
B:=\operatorname{Bar}(C),\qquad
N:=\operatorname{NonExc}(C).
\]
Thus $P$, $B$, and $N$ partition $[n]$, their cardinalities are
$r$, $k$, and $u$, respectively, and
\[
\eta_i<i\quad(\forall i\in P),
\qquad
\eta_i\geq i\quad(\forall i\in N).
\]
In particular, $B$ is the set of barred values of $\pi$,
not the set of its barred positions.

Write $B=\{b_1<\cdots<b_k\}$ and define
$\tau=\mathcal T(\pi)$ by
\begin{equation}\label{eq:comparison-negative-reversal}
\tau_i=
\begin{cases}
\eta_i,&i\notin B,\\
-\eta_{b_{k+1-a}},&i=b_a,\quad 1\leq a\leq k.
\end{cases}
\end{equation}
When $k=0$, the second case is absent.
We may equivalently construct $\tau$ as follows: regard barred letters as negative integers so that we may identify $\pi$ with a signed permutation in $B_n$,
take the inverse $\pi^{-1}\in B_n$ written in window notation, and reverse its subsequence
of negative entries while leaving all positive entries fixed.

Since $\eta\in\symm_n$, $\tau$ is indeed a signed permutation.
Its negative positions are exactly $B$. For $i\notin B$,
\[
\tau_i\geq i
\quad\Longleftrightarrow\quad
\eta_i\geq i\text{ and }i\notin B
\quad\Longleftrightarrow\quad
i\in N.
\]
Consequently,
\begin{equation}\label{eq:comparison-parameters}
\operatorname{neg}(\tau)=k,
\qquad
\operatorname{wex}(\tau)=u.
\end{equation}

\medskip
\noindent
\emph{Bijectivity.}
Given a signed permutation $\tau$ in the target set, recover
\[
B=\{i\in[n]:\tau_i<0\}=\{b_1<\cdots<b_k\}
\]
and define
\[
\eta_i=
\begin{cases}
\tau_i,&i\notin B,\\
|\tau_{b_{k+1-a}}|,&i=b_a.
\end{cases}
\]
The absolute values of $\tau$ are precisely $[n]$, so
$\eta\in\mathfrak S_n$.
Let $\sigma=\eta^{-1}$, and bar precisely the values in $B$
in the one-line notation of $\sigma$.
This recovers a bicolored permutation $\pi$ with $k$ bars.

Among the $n-k$ indices outside $B$, exactly $u$ satisfy
$\eta_i\geq i$, because $\operatorname{wex}(\tau)=u$.
The remaining $n-k-u=r$ satisfy $\eta_i<i$ and therefore
correspond to the unbarred excedances of $\pi$.
Thus $\exc{\pi}=r$.
Reversing the negative subsequence twice restores it, so this
construction and \eqref{eq:comparison-negative-reversal}
are mutually inverse.

\medskip
\noindent
\emph{Preservation of crossings.}
We expand the three crossing conditions according to the signs
of the relevant entries and compare their contributions with the five counts in Definition~\ref{def:cr}.

\begin{enumerate}
\item
Condition~\eqref{eq:comparison-crossing-upper} forces
$\tau_i,\tau_j>0$ and $i,j\in N$.
Its contribution is therefore
\[
\#\{(j_1,j_2)\in N^2:
j_1<j_2\leq\eta_{j_1}<\eta_{j_2}\},
\]
which is the $NN$ contribution to $\cro{\pi}$.

\item
Condition~\eqref{eq:comparison-crossing-mixed} forces
$\tau_i<0<\tau_j$ and hence $i\in B$. Since $i,j>0$, the inequality $-i<j$
is automatic. Its remaining inequalities are
\[
j\leq|\tau_i|<\eta_j,
\]
which also force $j\in N$.
The negative values $\tau_b$ satisfy
\[
\{|\tau_b|:b\in B\}=\{\eta_b:b\in B\}.
\]
Each value in either set occurs exactly once.
Hence this contribution is
\[
\#\{(b,j)\in B\times N:j\leq\eta_b<\eta_j\},
\]
the $BN$ contribution to $\cro{\pi}$.

\item
Suppose that
\eqref{eq:comparison-crossing-lower} holds with both
entries $\tau_i,\tau_j$ positive. Put $p_1=j$ and $p_2=i$.
The condition becomes
\[
\eta_{p_1}<\eta_{p_2}<p_1<p_2.
\]
These inequalities force $p_1,p_2\in P$.
Thus this contribution is precisely the $PP$ contribution
to $\cro{\pi}$.

\item
Suppose that
\eqref{eq:comparison-crossing-lower} holds with
$\tau_i>0>\tau_j$.
Writing $p=i$ and $b=j$, its condition becomes
\[
\eta_p<b<p.
\]
In particular, $p\in P$ and $b\in B$.
Conversely, every such pair satisfies
\eqref{eq:comparison-crossing-lower}.
This gives the $BP$ contribution to $\cro{\pi}$.

\item
Suppose that
\eqref{eq:comparison-crossing-lower} holds with both
entries negative, which forces $i,j\in B=\{b_1<b_2<\dots<b_k\}$. Write $j=b_a$ and $i=b_c$, where $a<c$.
The inequalities involving a positive position and a
negative entry are automatic, and the remaining condition is
\[
\tau_{b_c}>\tau_{b_a}.
\]
By \eqref{eq:comparison-negative-reversal}, this is equivalent to
\[
\eta_{b_{k+1-c}}<\eta_{b_{k+1-a}}.
\]
The map
\[
(a,c)\longmapsto(k+1-c,k+1-a)
\]
is a bijection on the pairs of indices with the first smaller
than the second. Therefore this contribution equals
\[
\#\{(b_d,b_e)\in B^2:
b_d<b_e,\ \eta_{b_d}<\eta_{b_e}\},
\]
the $BB$ contribution to $\cro{\pi}$.
\end{enumerate}

The remaining sign possibility in
\eqref{eq:comparison-crossing-lower},
namely $\tau_i<0<\tau_j$, is impossible.
Thus the five cases above exhaust all crossings of $\tau$, without
overlap, and reproduce exactly the five terms in
Definition~\ref{def:cr}.
This proves \eqref{eq:comparison-crossing-preserved}.

\medskip
\noindent
\emph{Preservation of alignments.}
Recall the alignment--crossing identity of Cho and Park in \eqref{eq:comparison-CP-complement}: for every $\tau\in \fB_n$,
\[
\begin{split}
\alSgn(\tau)
+\crSgn(\tau)
={}&
\bigl(n-\operatorname{wex}(\tau)\bigr)
\bigl(\operatorname{wex}(\tau)-1+\operatorname{neg}(\tau)\bigr)\\
&-\binom{\operatorname{neg}(\tau)}2.
\end{split}
\]
Using \eqref{eq:comparison-parameters} and $n=r+k+u$,
the right-hand side becomes
\begin{align*}
(n-u)(u-1+k)-\binom k2
&=(r+k)(u-1+k)-\binom k2\\
&=ru+ku+rk-r+\binom k2\\
&=\binom n2-\binom{r+1}2-\binom u2\\
&=\delta_{n,k,r}.
\end{align*}
On the other hand,
Corollary~\ref{cor:cr-al-complement-for-pi} gives
\[
\cro{\pi}+\al{\pi}
=\delta_{n,k,r}.
\]
Subtracting the already established equality
\eqref{eq:comparison-crossing-preserved} from these two
identities yields
\[
\alSgn(\mathcal T(\pi))
=\delta_{n,k,r}-\cro{\pi}
=\al{\pi}.
\]
This proves \eqref{eq:comparison-alignment-preserved}.

Finally, the negative positions of $\mathcal T(\pi)$ are
exactly the barred values of $\pi$.
Hence the value $1$ is unbarred in $\pi$ if and only if
$\mathcal T(\pi)_1>0$.
\end{proof}

\begin{corollary}\label{cor:comparison-joint-equidistribution}
For integers $n\ge 1$ and $0\le k\le n$, the joint distribution generating functions satisfy
\[
\sum_{\substack{
\pi\in\coloredperm{n}{k}\\
\exc{\pi}=r}}
z^{\cro{\pi}}q^{\al{\pi}}
=
\sum_{\substack{
\tau\in \fB_{n,k}\\
\operatorname{wex}(\tau)=n-k-r}}
z^{\crSgn(\tau)}
q^{\alSgn(\tau)}.
\]
The same identity holds after restricting the left-hand side
to bicolored permutations in which the value $1$ is unbarred and the
right-hand side to signed permutations with $\tau_1>0$.
Moreover, for every $\pi$ in the domain of $\mathcal T$,
\[
\ell(f_\pi)
=
r(n-k-r)+\alSgn(\mathcal T(\pi)).
\]
\end{corollary}

\begin{proof}
The generating-function identities follow by summing the
pointwise equalities in
Theorem~\ref{thm:comparison-signed-statistics}.
The affine-length identity follows from
Definition~\ref{def:al} and
\eqref{eq:comparison-alignment-preserved}.
\end{proof}

\section{Future directions}
\label{sec:future-directions}

A natural next step is to seek a representation-theoretic
model for $\coloredfunc{n}{k}{r}(\xxx;q)$ (resp. $\Phi^{\mathrm{unb}}_{n,k,r}(\xxx;q)$) that also explains
its algebraic structure.
For every choice of $n,k,r$ such that this symmetric function does not vanish, its degree-zero coefficient is
\[
[q^0]\coloredfunc{n}{k}{r}(\xxx;q)\text{ (resp. $[q^0]\Phi^{\mathrm{unb}}_{n,k,r}(\xxx;q)$)}=h_n(\xxx).
\]
Since $h_n$ is the Frobenius characteristic of the trivial
$\mathfrak S_n$-representation, this is compatible with the
degree-zero component of a connected graded algebra:
the one-dimensional space spanned by its identity element.

\begin{problem}\label{prob:graded-algebra-model}
For each triple $(n,k,r)$ for which the indexing set defining
$\coloredfunc{n}{k}{r}$ (resp. $\Phi^{\mathrm{unb}}_{n,k,r}$) is nonempty, construct a natural
finite-dimensional commutative $\mathbb C$-algebra
\[
R_{n,k,r}=\bigoplus_{d\geq0}R_{n,k,r}^{(d)},
\qquad
R_{n,k,r}^{(0)}=\mathbb C\cdot1,
\]
equipped with an action of $\mathfrak S_n$ by
degree-preserving algebra automorphisms, such that
\[
\sum_{d\geq0}q^d\Frob\bigl(R_{n,k,r}^{(d)}\bigr)
=
\coloredfunc{n}{k}{r}(\xxx;q)\text{ (resp. $\Phi^{\mathrm{unb}}_{n,k,r}(\xxx;q)$)}.
\]
The construction should be uniform in the parameters and
give an explicit combinatorial or geometric description of
the multiplication.
\end{problem}

The requirement of an explicit algebraic construction is
essential. Whenever a graded representation with the prescribed
Frobenius characteristic exists and its degree-zero component
is trivial, one can formally obtain a connected graded algebra
by declaring all products of positive-degree elements to be zero.
Such a square-zero construction does not explain the structure
underlying the symmetric function.

A particularly desirable solution would provide a presentation
of $R_{n,k,r}$ as a quotient of a polynomial ring by an explicit
$\mathfrak S_n$-stable homogeneous ideal.
The number of variables and their positive degrees need not
be prescribed in advance.
Such a model would give a representation-theoretic explanation
of Schur positivity and an algebraic interpretation of the
normalized affine-length grading.

A natural candidate is the Leray model of the Boolean matroid mentioned in the introduction whose graded Frobenius series is 
\[
    Q_{n}^{\mathsf{lab}}(\xxx;t,y)=\sum_{k,r}\left.\coloredfunc{n}{k}{r}(\xxx;q)\right|_{q=1}t^r y^k.
\]
One needs to find a further $q$-grading so that the triply graded Hilbert series of the refined Leray model of the Boolean matroid is given by 
\[
    \sum_{\pi\in\fB_n} y^{\nega{\pi}}t^{\exc{\pi}}q^{\al{\pi}}.
\]

\section{Acknowledgments}\label{sec:acknowledgements}
We would like to thank Alessandro Iraci, Martina Lanini, and Brendon Rhoades for discussions that led to the original question motivating this work, and Brendon Rhoades for bringing the question to our attention.

\section{Declaration of generative AI and AI-assisted technologies in the manuscript preparation process}
During the preparation of this work, the authors used GPT-5.6 sol for constructing the bijection in Section~\ref{sec:proofs-theta-formulas} via Lentfer's monomial codes, for assisting with LaTeX formatting, and for polishing the language. The authors reviewed and edited the output as needed and takes full responsibility for the content of the published article.

\printbibliography

\end{document}